\documentclass[a4paper,11pt]{smfart}

\usepackage{multirow,yfonts}
\usepackage{dsfont}
\usepackage[applemac]{inputenc}
\usepackage[francais]{babel}
\usepackage{smfthm}
\usepackage{amssymb,amsmath, amsfonts}
\usepackage{stmaryrd}
\usepackage{xypic}
\usepackage{graphicx}
\usepackage{chngpage}
\usepackage[toc,page]{appendix} 
\usepackage{subfigure}
\usepackage{color}

\definecolor{gris}{gray}{0.45}
 
  \newtheorem*{theo*}{Théorème}
  \newtheorem*{coro*}{Corollaire}
           
\newcommand {\Q}{\mathbb{Q}}

\newcommand{\Hhyp}{\mathbb{H}}
\newcommand {\N}{\mathbb{N}}

\newcommand{\kod}{\mathbf{kod}}
\newcommand{\F}{\mathcal{F}}
\newcommand{\calH}{\mathcal{H}}
\newcommand{\calM}{\mathcal{M}}
\newcommand{\calO}{\mathcal{O}}
\newcommand{\G}{\mathcal{G}}
\newcommand{\Ft}{\tilde{\mathcal{F}}}

\newcommand{\num}{\stackrel{num}{=}}
\newcommand{\calR}{\mathcal{R}}
\newcommand{\R}{\mathbb{R}}
\newcommand{\C}{\mathbb{C}}

\newcommand{\Z}{\mathbb{Z}} 
 
\newcommand{\calC}{\mathcal{C}} 

\newcommand{\Pd}{\mathbb{P}^2} 
\newcommand{\Pu}{\mathbb{P}^1} 
\newcommand{\fonction}[5]{$$\begin{array}{crcl} {#1} :&{#2}&\longrightarrow &   {#3} \\
                         &{#4}    &\longmapsto     &{#5}
                      \end{array}$$}

\begin{document}
\title[Un exemple de feuilletage modulaire]{Un exemple de feuilletage modulaire déduit d'une solution algébrique de l'équation de Painlevé VI}
\author{Ga\"el Cousin}
\date{}

\setcounter{tocdepth}{2}
\address{CNRS-IRMAR UNIVERSITÉ DE RENNES 1, CAMPUS DE BEAULIEU  35042 RENNES CEDEX}
\email{gael.cousin@univ-rennes1.fr}
\keywords{feuilletages holomorphes, dimension de Kodaira, surfaces modulaires de Hilbert, connexions plates, équation de Painlevé VI}

\vspace{-3 cm}
\maketitle
\vspace{-0.5 cm}

\vspace{0.5 cm}
\emph{\`A la mémoire de Marco Brunella.}\\

\begin{abstract}
     On peut construire facilement des exemples de connexions plates de rang $2$ sur $\Pd$ comme tirés en arrière de connexions sur $\Pu$.
On donne un exemple de connexion qui ne peut être obtenue de cette manière. Cet exemple est construit à partir d'une solution algébrique de l'équation de Painlevé VI. On en déduit un feuilletage modulaire. La preuve de ce fait repose sur la classification des feuilletages sur les surfaces projectives par leurs dimensions de Kodaira, fruit du travail de Brunella, McQuillan et Mendes.
On décrit ensuite le feuilletage dual.
Par une analyse fine de monodromie, on voit que notre surface bifeuilletée est revêtue par la surface modulaire de Hilbert construite en faisant agir $\mathrm{PSL}_2(\Z[\sqrt{3}])$ sur le bidisque.
\end{abstract}
\vspace{-0.5cm}
\section{Introduction}

Les feuilletages modulaires sont des feuilletages naturels sur les quotients du bidisque par des réseaux irréductibles $\Gamma$ de $\mathrm{PSL_2}(\R)\times\mathrm{PSL_2}(\R)$. Ils sont induits par les feuilletages horizontaux et verticaux du bidisque et sont naturellement munis de structures transversalement projectives. On peut se donner un tel $\Gamma$ par le choix d'un nombre réel quadratique, voir section $2.1$.

 Dans une classification récente de Brunella, McQuillan et Mendes, parmi les feuilletages sur les surfaces projectives, les feuilletages modulaires sont caractérisés par leurs dimensions de Kodaira : $\kod=-\infty$ et $\nu=1$ ; ce sont deux invariants numériques codant les propriétés de tangence du feuilletage.
 
Dans \cite{MR2142243}, Mendes et Pereira donnent les premiers exemples de modèles birationnels explicites pour des feuilletages modulaires. La découverte de ces feuilletages est fondée sur une bonne connaissance de la surface sous-jacente. Il apparaît que les structures transversalement projectives de deux de ces exemples (les feuilletages $\calH_2$ et $\calH_3$ associés à $\sqrt{5}$ dans \cite{MR2142243}) correspondent à des déformations isomonodromiques de feuilletages de Riccati à quatre pôles sur $\Pu \times \Pu \rightarrow \Pu$, c'est à dire à deux solutions de l'équation de Painlevé VI (PVI). Ces solutions sont, par construction, algébriques ; elles sont des transformées d'Okamoto de solutions icosahédrales de Dubrovin-Mazzocco \cite{MR1767271}, les solutions n°$31$ et $32$ de la liste de Boalch \cite{MR2254812}. 

L'objectif principal de cet article est de produire un exemple de feuilletage modulaire en inversant cette construction : le choix d'une solution de l'équation de Painlevé VI paramétrée par une courbe $\calC$ donne un feuilletage de Riccati $\calR$ sur un $\Pu$-fibré $P \rightarrow \Pu\times \calC$ ; en choisissant une section de ce fibré, on se donne un feuilletage transversalement projectif $\F$ sur la surface $\Pu\times \calC$ ; on souhaite que ce dernier soit un feuilletage modulaire.

 Le choix de la solution de (PVI) est bien sûr guidé par les propriétés des feuilletages modulaires : on a des contraintes sur la monodromie de $\calR$ et, d'après Touzet \cite[Théorème III$.2.6.$]{MR2008442}, $\F$ ne doit pas être tiré-en-arrière rationnel d'un feuilletage de Riccati $\calR_0$ sur une surface algébrique. Cherchant à obtenir un des feuilletages modulaires sur la surface rationnelle associée à $\sqrt{3}$, on a ainsi été conduit sans ambiguïté à la solution utilisée à la section \ref{non pb}.
On a trouvé un groupe de symétries d'ordre $4$ pour le feuilletage de Riccati $\calR$ construit à partir de cette solution. Soit $\tilde{\calR}$ le feuilletage de Riccati quotient. En choisissant une section du $\Pu$-fibré sous-jacent à $\tilde{\calR}$, on a obtenu le feuilletage $\F_{\omega}$ décrit ci-dessous.

\begin{theo} \label{thprinc}
  La surface bifeuilletée  $(\Pd,\{\F_{\omega},\G_{\tau}\})$ est un modèle birationnel d'une surface modulaire munie de ses feuilletages modulaires, où \\
  $\omega=-12y(1+3y)(3x-y)dx+\left [(10-18x)y^2-9x(18x-5)y-9x^2(9x-2) \right ]dy$ et\\
  $\tau=-12y(1+3y)(12x-y-3)dx\\+ \left [(4-18x)y^2-3(18x-1)(3x-1)y+9x(2-9x)(x-1)\right ]dy$.\\
  De plus, l'involution birationnelle de $\Pd$ suivante échange $\F$ et $\G$.
   \center{$\sigma : (x,y) \mapsto \left({\frac {3\,y \left( 3\,y+13 \right) x-y \left( 7\,y+9 \right) }{
 \left( 135\,y+9 \right) x-3\,y \left( 3\,y+13 \right) }},y \right).$}
\end{theo}
La preuve de ce théorème est fondée sur un calcul de dimensions de Kodaira.

\paragraph*{}
Après un revêtement double, on obtient les feuilletages modulaires de Hilbert construits en faisant agir $\Gamma=\mathrm{PSL}_2(\Z[\sqrt{3}])$ sur le bidisque (voir théorème~\ref{racine}). On peut en fait voir que la surface feuilletée du théorème \ref{thprinc} est celle construite à partir de l'extension de Hurwitz-Maass de 
$\mathrm{PSL}_2(\Z[\sqrt{3}])$, cf lemme~\ref{monodromie}.

Dans \cite[Theorem $2$ p $1273$]{MR2457528}, Corlette et Simpson établissent un résultat de factorisation pour certaines représentations $\rho : \pi_1(X) \rightarrow \mathrm{PSL}_2(\C)$ de groupes fondamentaux de variétés quasi-projectives : si la représentation ne se factorise pas par une courbe, alors elle est tiré-en-arrière  par une application $f : X \rightarrow Y$ d'une des représentations tautologiques d'un quotient $Y$ d'un polydisque. Ils manifestent leur intérêt pour la détermination d'une telle application $f$ pour les représentations de monodromie des feuilletages de Riccati obtenus à partir de solutions algébriques de (PVI). C'est ce que nous avons fait ici pour notre feuilletage initial $\calR$, avec $Y$ une surface modulaire.

Pour les propriétés générales des feuilletages sur les surfaces on se réfère à \cite{MR2114696}. Pour des propriétés particulières des feuilletages modulaires, on pourra consulter \cite{MR2142243}. Pour les notions relatives aux feuilletages transversalement projectifs, on recommande \cite{MR2337401} et \cite{MR2324555}. Pour une introduction aux connexions logarithmiques plates, voir \cite{yaknov}. Enfin, pour les calculs de groupes fondamentaux, on propose \cite{shimada} ou \cite{MR2830090}.
Par commodité, on a employé fréquemment un logiciel de calcul formel pour nos discussions. Toutefois, ce recours au calcul formel ne s'avère strictement nécessaire que pour ce qui est décrit à la section \ref{quotient}.   

L'auteur tient à remercier chaleureusement son directeur de thèse Frank Loray ainsi que Jorge Pereira pour leurs nombreuses indications. Remerciements également aux membres de l'équipe de géométrie analytique de l'IRMAR, à Philip Boalch, Serge Cantat, Slavyana Geninska et à Pierre Py.
On remercie aussi le CNRS pour son financement de thèse, l'IRMAR pour son accueil permanent et l'IMPA pour un séjour fructueux. Les commentaires du referee ont permis une grande amélioration de ce travail. Qu'il en soit remercié. 

\subsection{Surfaces modulaires, feuilletages modulaires}
Soit $K=\Q(\sqrt{d})$ avec $d\in \N^*$ sans facteur carré. Soit $\mathcal{O}_K$ l'anneau d'entiers de $K$.
Les deux plongements de $K$ dans $\R$ induisent deux plongements de $\mathrm{PSL_2}(\mathcal{O}_K)$ dans $\mathrm{PSL_2}(\R)$ : $\gamma \mapsto A$ et $\gamma \mapsto \bar{A}$, où $A \mapsto \bar{A}$ est l'action du groupe de Galois de $K$.
On obtient un plongement $ i : \mathrm{PSL_2}(\mathcal{O}_K) \rightarrow \mathrm{PSL_2}(\R)^2, \gamma \mapsto (A,\bar{A})$. Soit $\Gamma_K$ l'image de $i$.

D'après Baily-Borel \cite[Theorem 10.11]{MR0216035} Le quotient $X_K=(\Hhyp \times \Hhyp)/\Gamma_K$ est une surface complexe singulière qui se compactifie en une surface projective $Y_K$ par adjonction d'un ensemble fini de points (cusps) aux bouts de $X_K$.
Les singularités de $X_K$ sont des singularités "quotients cycliques", on les appelle aussi singularités de Hirzebruch-Jung. Leur désingularisation a été donnée par Hirzebruch \cite{MR0062842}, voir aussi \cite[\S IV]{MR570310}. Cette désingularisation consiste à remplacer chaque point singulier par une chaîne adéquate de courbes rationnelles lisses : une chaîne de Hirzebruch-Jung.

Chaque cusp donne aussi lieu à une singularité de $Y_K$, la désingularisation des cusps à été expliquée par Hirzebruch dans \cite[\S 2]{MR0393045}.
Cette fois, chaque singularité est remplacée par un cycle (contractile) de courbes rationnelles lisses. On notera $Z_K$ ou $Z_{\sqrt{d}}$ la surface projective lisse ainsi obtenue, cette surface est appelée la surface modulaire de Hilbert associée à $K$.

Les surfaces $Z_{\sqrt{d}}$ ont été intensément étudiées, par exemple Hirzebruch, Van de Ven et Zagier \cite{MR0480356} ont démontré que $Z_{\sqrt{d}}$ est rationnelle exactement pour $d \in \{ 2,3,5,6,7,13,15,17,21, 33\}.$
Il existe aussi des tables donnant le nombre et le type des singularités de Hirzebruch-Jung en fonction de $d$. Le nombre de cusp est le nombre de classes de $K$, pour une exposition détaillée des propriétés de $Z_{\sqrt{d}}$, se référer à \cite{MR930101}.

Dans son article \cite{MR2435846}, McQuillan considère des surfaces un peu plus générales.
\begin{defi}\label{modulaire}
Soit $\Gamma$ un réseau de $\mathrm{PSL}_2(\R)^2$ non commensurable à un produit $\Gamma_1 \times \Gamma_2$ de sous-groupes de $\mathrm{PSL}_2(\R)$. La surface $X_{\Gamma}=(\Hhyp~\times~\Hhyp)/\Gamma$ est compacte ou peut se compactifier en une surface projective comme les surfaces précédentes, c.-à-d. par adjonction, à chaque bout, d'un cycle contractile de courbes rationnelles lisses. Les seules singularités de la surface ainsi compactifiée sont alors des singularités de Hirzebruch-Jung. Après leur désingularisation, on obtient une surface lisse qu'on appelle surface modulaire et note $Z_\Gamma$.
\end{defi}
\begin{rema} Dire que deux sous groupes $H,H'$ d'un groupe $G$ sont commensurables signifie qu'il existe $g\in G$ tel que $H\cap gH'g^{-1}$  soit d'indice fini dans $H$ et dans  $gH'g^{-1}$.

D'après les travaux de Margulis sur l'arithméticité des réseaux, dans le cas où $X_{\Gamma}$ n'est pas compact, le groupe $\Gamma$ est commensurable à un $\Gamma_K$ (dans $\mathrm{PSL}_2(\R)^2$).
\end{rema}

\begin{defi}
Les feuilletages modulaires de $Z_\Gamma$ sont les images des feuilletages verticaux et horizontaux de $\Hhyp^2$. On les note $(\F_{\Gamma},\G_{\Gamma})$. Si $Z_\Gamma=Z_K$ est une surface modulaire de Hilbert, alors $\F_{\Gamma}$ et $\G_{\Gamma}$ sont  appelés feuilletages modulaires de Hilbert et notés $\F_K$ et $\G_K$.
\end{defi}
La principale source d'intérêt pour les feuilletages modulaires et leur rôle dans la classification birationnelle des feuilletages, nous y reviendrons à la section \ref{strategie}. 

Les feuilletages modulaires sont naturellement munis de structures transversalement projectives. 
\subsection{Feuilletages transversalement projectifs, feuilletages de Riccati}
La définition suivante est dûe à \cite{MR2337401} et équivalente à celle de \cite{MR1432053}.
\begin{defi}
Soit $\calH$ un feuilletage de codimension $1$ sur une variété complexe lisse $M$. Une \textbf{structure transversalement projective} (singulière) pour $\calH$ est la donnée d'un triplet $(\pi,\calR,\sigma)$ consistant en
\begin{enumerate}
\item  \label{un} un $\Pu$-fibré   $\pi : P \rightarrow M$;
\item \label{deux} un feuilletage holomorphe singulier $\calR$ de codimension $1$ sur $P$ transverse à la fibre générale de $\pi$ et 
\item une section méromorphe  $\sigma$ de $\pi$  telle que $\calH=\sigma^*\calR$.
\end{enumerate}
En présence des conditions \ref{un} et \ref{deux}, on dit que $\calR$ est un \textbf{feuilletage de Riccati} sur le fibré $\pi$.
\end{defi}
\begin{defi}
Soient $\calR$ un feuilletage de Riccati sur $\pi$, $D$ est un diviseur de $M$ tel que $\calR_{\vert M \setminus D}$ est transverse aux fibres de $\pi_{\vert M \setminus D}$ et $* \in M\setminus D$. Par compacité des fibres, en relevant les lacets de $M \setminus D$ dans les feuilles de $\calR$, on peut définir une représentation $\pi_1(M \setminus D,*)\rightarrow Aut(\pi^{-1}(*))$. Si $D$ est le plus petit diviseur ayant cette propriété, on l'appelle le \textbf{lieu polaire} de $\calR$ et la représentation est appelée \textbf{représentation de monodromie} de $\calR$, l'image de cette représentation est le \textbf{groupe de monodromie} de $\calR$.
\end{defi}
En pratique, on doit choisir une coordonnée sur $\pi^{-1}(*)$ et la monodromie est donnée par des éléments de $\mathrm{PSL}_2(\C)$.

\begin{defi}\label{reduction}
Soient $\calC_1\subset \calC_2$ deux hypersurfaces sur  une variété complexe lisse $M$ et $G$ un groupe.
Si $\rho : \pi_1(M\setminus \calC_2)\rightarrow G$ se factorise par $i_* :\pi_1(M\setminus \calC_2)\rightarrow \pi_1(M\setminus \calC_1)$ : $\rho=\rho_1 i_*$, on dit que $\rho$ se réduit à $M\setminus \calC_1$ et que $\rho_1$ est une réduction partielle de $\rho$. Si $\calC_1$ est minimal pour l'inclusion avec cette propriété, on dit que $\rho_1$ est la réduction de $\rho$ et que $\rho_1$ est réduite. Si, de plus, $M$ est simplement connexe, on appelle $\calC_1$ le support de $\rho$. 
\end{defi}

Le résultat suivant explique notre intérêt pour la notion de feuilletage transversalement projectif.
\begin{lemm}\label{lemmemonodmodulaire} Soit $Z_{\Gamma}$ une surface modulaire. 
Soit $\mathcal{V}$ le complémentaire dans $Z_{\Gamma}$ des cycles de courbes rationnelles et des chaînes de Hirzebruch-Jung apparus dans le processus de compactification/désingularisation de $X_{\Gamma}$.
La projection $\Hhyp \times \Hhyp \rightarrow X_{\Gamma}$ donne, par restriction, le revêtement universel de $\mathcal{V}$: $\mathcal{U} \rightarrow \mathcal{V}$.
On y associe sa représentation tautologique $\tau : \pi_1(\mathcal{V})\rightarrow \Gamma$.

Soit $\tau_i$ la projection de $\tau$ sur le $i$-ème facteur de $\mathrm{PSL_2}(\R)^2$. Les feuilletages $\F_{\Gamma \vert \mathcal{V}}$ et $\G_{\Gamma \vert \mathcal{V}}$ sur $\mathcal{V}$ possèdent des structures transversalement projectives de monodromies respectives $\tau_1$ et $\tau_2$.

\begin{proof}
Il suffit de faire la preuve pour $\F_{\Gamma}$. Soient  $(u_1,u_2)$ les coordonnées naturelles sur $\mathcal{U}\subset \Hhyp \times \Hhyp$. Soit $\sigma=\{z=u_1\}$ une section de $\Pu \times  \mathcal{U}\rightarrow \mathcal{U}$.  En fixant $(z,u)\cdot \gamma=(\gamma_1^{-1}\cdot z  ,\gamma^{-1} \cdot u)$ pour tout $\gamma=(\gamma_1,\gamma_2) \in \mathrm{PSL}_2(\R)^2$ et  tout $(z,u) \in \Pu \times \mathcal{U}$, on définit une action proprement discontinue de $\Gamma$ sur $\Pu \times  \mathcal{U}$. Le quotient $P=(\Pu \times  \mathcal{U})/\Gamma$ est muni d'une structure de $\Pu$-fibré $\pi : P \rightarrow X$ dont une section holomorphe $\tilde{\sigma}$ est induite par $\sigma$. De surcroît, le feuilletage de Riccati sur $\Pu \times  \mathcal{U}$ défini par $dz=0$ passe au quotient et fournit un feuilletage de Riccati $\calR$ de telle sorte que $(\pi,\calR,\tilde{\sigma})$ soit une structure transversalement projective pour la restriction de $\Ft_{\Gamma}$ à $X$. Par construction cette structure a pour monodromie la projection $\tau_1$ de la représentation tautologique du quotient  $\mathcal{U}/\Gamma$.
\end{proof}
\end{lemm}
\begin{rema} 
 On peut prolonger cette structure transverse (de façon singulière) à tout $Z_{\Gamma}$. Cela ne sera pas utilisé ici, on remet cette discussion à un prochain travail.
\end{rema}
On définit une relation d'équivalence naturelle entre feuilletages de Riccati.
\begin{defi}
Soit $\calR$ et $\calR'$ deux feuilletages de Riccati sur $\Pu\times M\rightarrow M$. On dit que $\calR$ et $\calR'$ sont \textbf{biméromorphiquement équivalents} s'il existe une \textbf{transformation de jauge méromorphe} \fonction{\phi}{\Pu\times M}{\Pu\times M}{(z,x)}{(A(x) \cdot z,x)} avec  $A: M \dasharrow \mathrm{GL_2}(\C)$ méromorphe, de sorte qu'on ait $\phi^{*}\calR=\calR'$.
\end{defi}
On note alors que les réductions des représentations de monodromie de $\calR$ et $\calR'$ sont identiques.

Nous prendrons aussi en considération la situation suivante.

\begin{defi} Soient $M$ et $M'$ deux variétés complexes projectives lisses. 
On dit qu'un feuilletage de Riccati $\calR$ sur le $\Pu$-fibré $\pi:\Pu \times M \rightarrow M$  est \textbf{tiré-en-arrière} (rationnel) d'un feuilletage de Riccati $\calR_0$ sur $\pi':\Pu \times M' \rightarrow M'$ s'il existe une application rationnelle dominante $\phi: M \dasharrow M'$ telle que $\calR=(Id\times \phi)^*\calR_0$. Dans le cas où $M'$ est une courbe algébrique, on dit que $\calR$ est tiré-en-arrière d'un feuilletage de Riccati au dessus d'une courbe.
\end{defi}
Des feuilletages de Riccati apparaissent naturellement par projectivisation de connexions plates : soit $\nabla$ une connexion plate sur $E$ un fibré de rang $2$,  le feuilletage donné par ses sections horizontales induit un feuilletage de Riccati $\calR_{\nabla}$ sur $\mathbb{P}(E)$. On voit aisément que la monodromie de $\calR_{\nabla}$ est alors déduite de celle de $\nabla$ par composition avec la projection $\mathbb{P} : \mathrm{GL}_2(\C)\rightarrow \mathrm{PGL}_2(\C)=\mathrm{PSL}_2(\C)$. La trace de $\nabla$ et $\calR_{\nabla}$ déterminent $\nabla$.


Pour se donner une connexion plate de rang $2$, on peut utiliser une solution algébrique de l'équation de Painlevé VI.

\vspace{-0.205cm}
\subsection{Équation de Painlevé VI et isomonodromie}\label{PVIisom}
 L'équation de Painlevé VI de paramètres $(\theta_0, \theta_1, \theta_t,\theta_{\infty})$ est l'équation différentielle non linéaire du second ordre suivante, on la note (PVI)$_{\theta}$.
 \footnotesize
 $$\frac{d^2q}{d^2t}=\frac{1}{2}\left(\frac{1}{q}+\frac{1}{q-1}+\frac{1}{q-t}\right)\left(\frac{dq}{dt}\right)^2-\left(\frac{1}{t}+\frac{1}{t-1}+\frac{1}{q-t}\right)\frac{dq}{dt}$$ 
 $$ + \frac{q(q-1)(q-t)}{2t^2(t-1)^2}\left((\theta_{\infty}-1)^2
 -\theta_0^2\frac{t}{q^2}+\theta_1^2\frac{t-1}{(q-1)^2}+(1-\theta_t^2)\frac{t(t-1)}{(q-t)^2}\right)$$
 \normalsize
Si $t_0\in \C\setminus \{0,1\}$, les solutions $q(t), t\in (\C,t_0)$ de (PVI) gouvernent les germes de déformations isomonodromiques  à un paramètre $(\nabla_t)_{t \in (\C,t_0)}$ de connexions de rang deux de trace nulle avec $4$ pôles simples mobiles : $x=0,1,t,\infty$ sur \\$\C^2\times\Pu \rightarrow \Pu$. On entend par isomonodromie que les connexions plates $(\nabla_t)_{t\in (\C,t_0)}$ sont les restrictions d'une connexion plate à pôles simples $\nabla$ sur $p: \C^2 \times (\C,t_0)\times \Pu \rightarrow (\C,t_0)\times \Pu $. Sous des hypothèses de généricité sur les valeurs propres $(\pm \theta_i/2)$ des résidus de $\nabla_t$ en $x=i$, $i=0,1,t,\infty$, il y a une bijection entre les telles déformations (modulo transformations holomorphes de fibrés vectoriels) et les solutions de l'équation (PVI)$_{\theta}$, cf \cite[section 3.4]{MR1118604}.
Cela peut être exprimé par les formules suivantes. 

Soit $q(t)$ une solution de (PVI)$_{\theta}$ et $$p=\frac{1}{2}\left( \frac{t(t-1)}{q(q-1)(q-t)}\frac{dq}{dt}+\frac{\theta_0}{q}+\frac{\theta_1}{q-1}+\frac{\theta_t-1}{q-t}\right).$$
Alors, à transformation holomorphe de fibré près, la connexion plate $\nabla$ correspondante est donnée par
\fonction{\nabla_q} {\calO\oplus \calO}{\calM^1\oplus \calM^1}{ \left [ \begin{array}{c} z_1\\z_2 \end{array} \right]}{\left [ \begin{array}{c} dz_1\\dz_2 \end{array} \right ] -\left [ \begin{array}{cc} \beta_q/2&\alpha_q\\-\gamma_q&-\beta_q/2 \end{array} \right ]\cdot\left [ \begin{array}{c} z_1\\z_2 \end{array} \right];}
où les 1-formes méromorphes $\alpha_q, \beta_q,\gamma_q \in \calM^1$ sont données par les formules ci-dessous.
 \footnotesize
 $$\rho=1-\frac{1}{2}(\theta_0+\theta_1+\theta_t+\theta_{\infty});$$
$$\alpha=\left(-{\frac { \left( -\rho-{\it \theta_0}-{\it \theta_1}+qp+1-p \right) 
 \left( -{\it \theta_0}+1-{\it \theta_1}-p-{\it \theta_t}+qp-\rho \right) {t
}^{2}}{ \left( -x+t \right)  \left( -1+\rho+{\it \theta_0}+{\it \theta_1}+
{\it \theta_t}-qp+pt \right)  \left( -1+t \right) }}\right.$$ $$+{\frac { \left( -{
\it \theta_0}+1-{\it \theta_1}-p-{\it \theta_t}+qp-\rho \right)  \left( {q}^
{2}p-qp-\rho\,q+q-q{\it \theta_0}-q{\it \theta_t}+{\it \theta_t}-q{\it 
\theta_1} \right) t}{ \left( -x+t \right)  \left( -1+\rho+{\it \theta_0}+{
\it \theta_1}+{\it \theta_t}-qp+pt \right)  \left( -1+t \right) }}+$$

$$\left.{\frac { \left( -p-{\it \theta_1}+qp \right) t+q{\it \theta_0}-{q}^{2}p+qp
-{\it \theta_t}+\rho\,q+q{\it \theta_1}+1-\rho-q+q{\it \theta_t}-{\it \theta_0
}}{ \left( x-1 \right)  \left( -1+t \right) }}\right)dx$$

$$+\left({\frac { \left( -\rho-{\it \theta_0}-{\it \theta_1}+qp+1-p \right) 
 \left( -{\it \theta_0}+1-{\it \theta_1}-p-{\it \theta_t}+qp-\rho \right) {x
}^{2}}{ \left( x-1 \right)  \left( xp-1+\rho+{\it \theta_0}+{\it \theta_1}
+{\it \theta_t}-qp \right)  \left( -x+t \right) }}\right.$$

$$-{\frac { \left( -{\it \theta_0}+1-{\it \theta_1}-p-{\it \theta_t}+qp-\rho
 \right)  \left( {q}^{2}p-qp-\rho\,q+q-q{\it \theta_0}-q{\it \theta_t}+{
\it \theta_t}-q{\it \theta_1} \right) x}{ \left( x-1 \right)  \left( xp-1+
\rho+{\it \theta_0}+{\it \theta_1}+{\it \theta_t}-qp \right)  \left( -x+t
 \right) }}$$
 $$\left.+{\frac { \left( -1+{\it \theta_0}+\rho+{\it \theta_1} \right)  \left( -{
\it \theta_0}+1-{\it \theta_1}-p-{\it \theta_t}+qp-\rho \right) px}{ \left( 
xp-1+\rho+{\it \theta_0}+{\it \theta_1}+{\it \theta_t}-qp \right)  \left( -1
+\rho+{\it \theta_0}+{\it \theta_1}+{\it \theta_t}-qp+pt \right) }}\right.$$
$$ \left. -{\frac {
 \left( q-1 \right)  \left( -{\it \theta_0}+1-{\it \theta_1}-p-{\it \theta_t
}+qp-\rho \right) x}{ \left( x-1 \right)  \left( -1+t \right) }}\right)dt;$$

\newpage
$$\beta=\left( {\frac {{\it \theta_0}}{x}}+{\frac { \left( 2\,p+2\,\rho-2\,qp-2
+2\,{\it \theta_0}+2\,{\it \theta_1}+{\it \theta_t} \right) t+2\,q+{\it 
\theta_t}-2\,qp-2\,\rho\,q+2\,{q}^{2}p-2\,q{\it \theta_1}-2\,q{\it \theta_t}
-2\,q{\it \theta_0}}{ \left( -1+t \right)  \left( -x+t \right) }}\right.$$$$\left.+{
\frac { \left( 2\,p-2\,qp+{\it \theta_1} \right) t-2\,q{\it \theta_0}+2\,
\rho+{\it \theta_1}-2+2\,{\it \theta_0}+2\,{q}^{2}p-2\,q{\it \theta_1}-2\,qp
-2\,\rho\,q+2\,q+2\,{\it \theta_t}-2\,q{\it \theta_t}}{ \left( x-1
 \right)  \left( -1+t \right) }} \right) dx  $$
 
  $$+ \left( -
{\frac {q \left( 1-\rho-{\it \theta_0}-{\it \theta_1}-{\it \theta_t}+qp
 \right) }{t}}  +{\frac {p \left( {\it 
\theta_1}+\rho+{\it \theta_0}-1 \right) }{-1+\rho+{\it \theta_0}+{\it \theta_1
}+{\it \theta_t}-qp+pt}}\right.$$ 
 
 $$+{\frac { \left( -2\,p-2\,\rho+2\,qp+2-2\,{\it \theta_0}-2
\,{\it \theta_1}-{\it \theta_t} \right) x+2\,qp+2\,\rho\,q-2\,{q}^{2}p+2\,
q{\it \theta_1}+2\,q{\it \theta_t}+2\,q{\it \theta_0}-2\,q-{\it \theta_t}}{
 \left( x-1 \right)  \left( -x+t \right) }}$$

$$\left.+{\frac { \left( q-1 \right)  \left( -{\it 
\theta_0}+1-{\it \theta_1}-p-{\it \theta_t}+qp-\rho \right) x+ \left( q-1
 \right)  \left( -{\it \theta_0}+1-{\it \theta_1}-p-{\it \theta_t}+qp-\rho
 \right) }{ \left( x-1 \right)  \left( -1+t \right) }} \right) dt;$$

$$\gamma={\frac { \left( -1+\rho+{\it \theta_0}+{\it \theta_1}+{\it \theta_t}-qp+pt
 \right)  \left( x-q \right) dx }{ \left( x-t \right) 
x \left( x-1 \right) }}-{\frac { \left( -q+t \right)  \left( -1+\rho+{
\it \theta_0}+{\it \theta_1}+{\it \theta_t}-qp+pt \right) dt}{ \left( x-t \right) t \left( -1+t \right) }}.$$
\normalsize

\begin{rema}
Pour éviter les problèmes de coquilles pour les éventuels utilisateurs de ces formules, on donne ces dernières dans une feuille de calcul Maple sur la page web de l'auteur.
\end{rema}

Le feuilletage de Riccati $\calR$ sur $\mathbb{P}(\C^2)\times \left ((\C,t_0)\times \Pu\right)$ déduit par projectivisation de $\nabla_q$ est donné par la 1-forme $-dz+\alpha_q+\beta_q z+\gamma_q z^2$, où $z\in \C\cup \{\infty\}$ satisfait $z=z_1/z_2$. 
La coordonnée $z$ est choisie de sorte que, 
\begin{itemize}
\item pour tout $t_1$ voisin de $t_0$, la restriction $\calR_{t_1}$ de $\calR$ à $t=t_1$ possède une singularité en $s_i=\{(x,z)=(i,i)\}$, pour $i=0,1, \infty$.
\item l'indice de Camacho-Sad pour toute section locale invariante passant par  $s_i$ est $\theta_i$.
\end{itemize}
 Après cette normalisation seuls deux paramètres déterminent $\calR_{t_1}$ : $p(t_1)$ et $q(t_1)$ ; $q(t_1)$ est déterminé par le fait que $(x,z)=(q(t_1),\infty)$ et $s_{\infty}$ sont les seuls points de tangence entre $\calR_{t_1}$ et $z=\infty$.

Soit $q(t)$ une fonction algébrique qui satisfait l'équation (PVI)$_{\theta}$.
Si $(q(s),t(s)): \calC \rightarrow \Pu \times \Pu$ est une paramétrisation de cette solution, par les formules ci-dessus, on en déduit encore un feuilletage de Riccati sur le $\Pu$-fibré trivial au dessus de $\calC\times \Pu$; son lieu polaire est donné par la réunion de $x=0,x=1,x=t(s),x=\infty$ et de fibres de $\calC\times \Pu \rightarrow \calC$ qui contiennent des intersections entre $x=t(s)$ et $x=0,x=1,x=\infty$.

\subsection{Stratégie}\label{strategie}
Dans la section \ref{construction_et_etude}, nous allons obtenir un feuilletage transversalement projectif à l'aide d'une solution algébrique de (PVI) et montrer qu'il s'agit d'un feuilletage modulaire. 

Pour ce faire, nous emploierons la classification birationnelle des feuilletages sur les surfaces par Brunella, McQuillan et Mendes. Nous en rappelons ici les aspects que nous utiliserons; tout cela est décrit dans \cite[chapters 8 and 9]{MR2114696} à l'exception du théorème \ref{thmnonabondance}, cf \cite{MR1965362} et \cite{MR2435846}. 

Cette classification fait intervenir deux invariants birationnels qu'on peut associer à un feuilletage holomorphe $(X,\F)$ sur une surface projective lisse : sa dimension de Kodaira $\kod(\F)$ et sa dimension de Kodaira numérique $\nu(\F)$.

Soit $(\tilde{X},\Ft)$ une désingularisation de $\F$, et $v$ un champs de vecteurs méromorphe sur $X$ qui engendre $\Ft$. Le diviseur canonique de $\Ft$ est défini comme $K_{\Ft}=\calO_{\tilde{X}}(v_{\infty}-v_0)$ où $v_{\infty}$ est le diviseur des pôles de $v$ et $v_0$ celui de ses zéros.
si $K_{\Ft}$ est pseudo-effectif, alors on peut calculer sa décomposition de Kodaira numérique : $K_{\Ft}\num P+N$ ; où $P$ est un $\Q$-diviseur nef et $N$ est un $\Q^+$-diviseur contractile dont les composantes irréductibles sont orthogonales à $P$, comme décrit dans \cite[pp~101-102]{MR2114696}.
La dimension de Kodaira numérique de $\F$ est alors définie par :
\begin{equation*} \nu(\F):=\begin{cases}
 0 &\mbox{si } P\num 0,\\
1& \mbox{si }  P\stackrel{num}{\neq}0 \mbox{ et } P\cdot P=0,\\
2 & \mbox{si } P\cdot P>0.
\end{cases}
\end{equation*}
Si $K_{\Ft}$ n'est pas pseudo-effectif, on pose $\nu(\F):=-\infty$.

La dimension de Kodaira de $\F$ est $$\kod(\F):=\mathrm{limsup}_n \frac{\mathrm{log}(h^0(\tilde{X}, K_{\Ft} ^n))}{\mathrm{log}(n)}\in\{2,1,0,-\infty\}.$$

On a pour tout feuilletage $\F$ : \begin{equation} \label{ineg} \nu(\F)\geq \kod(\F). \end{equation}
 On s'intéresse au cas $\nu=1$.
Dans ce cadre, répétons les résultats qui décrivent $\F$ en fonction de $\kod(\F) \in \{1,0,-\infty\}$. 

\begin{theo}[McQuillan-Mendes] \label{McMen}
Si $\F$ est un feuilletage réduit sur une surface projective lisse et $\kod(\F)=1$ alors $\F$ est 
\begin{enumerate}
\item un feuilletage turbulent,
\item une fibration elliptique non-isotriviale,
\item \label{trois}une fibration isotriviale de genre $g \geq 2$ ou
\item un feuilletage  de Riccati.
\end{enumerate}
\end{theo}

\begin{theo}[McQuillan]\label{Mc}
Si $\F$ un feuilletage réduit sur une surface projective lisse tel que $\kod(\calH)=0$ alors $\nu(\calH)=0$.
\end{theo}

\begin{theo}[Brunella-McQuillan]\label{thmnonabondance}
Soit $\F$ un feuilletage holomorphe sur une surface projective lisse avec $\nu(\F)=1$ et $\kod(\F)=-\infty$ alors, à transformation birationnelle près, $\F$ est un feuilletage modulaire.
\end{theo}

Nous déduisons de cela l'énoncé suivant.
\begin{prop}\label{cor}
Soit  $\F$ un feuilletage réduit transversalement projectif sur une surface  projective $X$, tel que $\nu(\F)=1$, dont une structure transverse $(\pi: \Pu \times X \rightarrow X,\calR,\sigma)$ a une monodromie non virtuellement abélienne. On suppose aussi que $\calR$ n'est pas birationnellement équivalent au tiré-en-arrière d'un feuilletage de Riccati au dessus d'une courbe. Alors, à transformation birationnelle près, $\F$ est un feuilletage modulaire.

\begin{proof}
D'après les énoncés précédents, il nous suffit d'exclure les éventualités $1-4$ données  dans le théorème $\ref{McMen}$. Pour ce faire, on utilise le lemme \ref{SLP} qui montre que $\F$ a une unique structure transversalement projective au sens de \cite{MR2337401} et procède au cas par cas:
\begin{enumerate}

\item Si $\F$ est un feuilletage turbulent, alors il admet une structure transversalement projective à monodromie virtuellement abélienne (donnée par des automorphismes de la courbe elliptique sous-jacente), ce qui est exclu.
\item et \ref{trois}. Il n'est pas donné par une fibration, sans quoi il aurait une intégrale première méromorphe $f$ qui lui fournirait une structure de feuilletage transversalement euclidien : $dz=df$, de monodromie triviale; ce qui n'est pas conforme à nos hypothèses.
\setcounter{enumi}{3}
\item Le feuilletage $\F$ ne peut être un feuilletage de Riccati, puisque les feuilletages de Riccati sur $X$ ont une structure transverse donnée par une feuilletage de Riccati birationnellement équivalent au tiré-en-arrière d'un feuilletage de Riccati au dessus d'une courbe, cf \cite[section $3.1.$]{MR2337401}.
\end{enumerate}
\end{proof}
\end{prop}

\begin{lemm}[Loray-Pereira] \label{SLP}
Soit $X$ une  variété complexe projective lisse.
Si $\F$ est un feuilletage de codimension $1$ sur $X$ qui possède deux structures transversalement projectives $\Sigma=(\Pu \times X \rightarrow X,\calR,\sigma)$ et $\Sigma_0$ non-équivalentes au sens de \cite{MR2337401}, alors le groupe de monodromie de $\Sigma$ contient un groupe abélien d'indice $\leq 2$  ou $\calR$ est birationnellement équivalent  à un feuilletage de Riccati tiré-en-arrière d'un feuilletage de Riccati au dessus d'une courbe. 
\begin{proof}
La preuve utilise \cite[Proposition 2.1]{MR1432053}, voir \cite[Lemma $5.4.$]{MR2337401}. Si on demande $X$ projective dans les hypothèses, c'est notamment afin de pouvoir trivialiser birationnellement le $\Pu$-fibré associé à $\Sigma_0$.  Une bonne partie de cette énoncé persiste si $X$ est une variété complexe lisse et le fibré sous-jacent à $\Sigma_0$ est trivial ; nous y reviendrons au lemme \ref{lemclean}. 
\end{proof}
\end{lemm}

Ainsi nous démontrerons (théorème $\ref{theomod}$) que le feuilletage $\F$ construit à la section $2$ est un feuilletage modulaire $\F_{\Gamma}$ en appliquant la proposition \ref{cor}. 
Dans la section \ref{raffinements} nous comprendrons plus en détails le groupe $\Gamma$.

\section{Construction et étude d'un exemple}\label{construction_et_etude}
\subsection{ Une solution de l'équation de Painlevé VI} \label{non pb}
Considérons la solution algébrique de (PVI) suivante. Elle est paramétrée par $s \in \calC=\Pu$.\\
$$q(s)={\frac { \left( {s}^{6}+15\,{s}^{4}-5\,{s}^{2}+45 \right) s \left( s+1
 \right)  \left( s-3 \right) ^{2}}{ \left( 5\,{s}^{6}-5\,{s}^{4}+135\,
{s}^{2}+81 \right)  \left( s+3 \right)  \left( s-1 \right) ^{2}}},
 t(s)=-{\frac { \left( s+1 \right) ^{3} \left( s-3 \right) ^{3}}{ \left( s-1
 \right) ^{3} \left( s+3 \right) ^{3}}},$$
$$\theta_0=-\frac{5}{6}, \theta_1=-1, \theta_t=-1, \theta_{\infty}=\frac{1}{6}.$$
C'est l'image de la solution tétraèdrale n°6 de Boalch \cite{MR2322328} par une transformation d'Okamoto, la transformation
 $s_{\delta} \circ s_{ \infty}$ dans les notations de  \cite{LiTy}. Une solution équivalente sous le groupe d'Okamoto a d'abord été donnée par  \cite{MR1911252}.

Grâce aux formules de la section \ref{PVIisom} on obtient un feuilletage de Riccati $\calR$ sur $\mathbb{P}(\C^2) \times \calC \times \Pu$ à partir de cette solution. On notera $(\calR_s)$ la famille des restrictions de $\calR$ aux niveaux de la projection $\mathbb{P}(\C^2) \times \calC \times \Pu \rightarrow \calC$.
\subsubsection{Monodromie de $\calR_s$} \label{monodlin}
On veut décrire la représentation de monodromie $r_1$ de $\calR_{s_0}$, pour $s_0\in \calC$ général. 
 Soient $x \in \C \cup \{\infty\}$ une coordonnée sur $\Pu$. Soit $\Pu_{t(s_0)}=\Pu\setminus \{x=0,1,{t(s_0)},\infty\}$ et $*\in \Pu_{t(s_0)}$. 
Choisissons $v,w,t,u \in \pi_1(\Pu_{t(s_0)},*)$ des lacets simples faisant le tour dans le sens direct de $x=0,1,{t(s)},\infty$ respectivement et tels que $tuvw=1$.

Soit $r: \pi_1(\Pu_{t(s_0)},*)\rightarrow \mathrm{PSL}_2(\C)$.
Soient $T,U,V$ et $W$ les images respectives de $t,u,v$ et $w$ par un relèvement de $r$ à $\mathrm{SL}_2$. 
Si $r$ est irréductible (i.e. ne possède pas de point fixe sur $\Pu$), alors elle est déterminée, modulo conjugaison globale, par les traces de  $T,U,V,W, VW,WT$ et $VT$.
Soit $(\mathcal{T}_s)_{s \in \calC}$ la famille de feuilletages de Riccati associée à la solution tétahédrale n°6 de Boalch et $r_2 :\pi_1(\Pu_{t(s_0)},*)\rightarrow \mathrm{PSL}_2(\C)$ la représentation de monodromie de $\mathcal{T}_{s_0}$.

  Dans \cite[Tab. $2$ p 93]{MR2322328}, Boalch donne les traces de matrices qui correspondent à $r_2$ modulo l'action de  $\mathcal{MCG}(\mathbb{S}_4^2)$, le mapping class group de la sphère épointée quatre fois $\mathbb{S}_4^2$. L'action de $\mathcal{MCG}(\mathbb{S}_4^2)$ correspond à l'ambiguïté sur le choix des lacets  simples $t,u,v,w$ ci-dessus, voir \cite[Theorem $1.9$ p $30$]{MR0425944}.
Le théorème \cite[Theorem $2.3$ p $6$]{MR2036953} (voir une autre preuve dans \cite[p $202$]{MR2107041})
dit que les traces des images de $vw,wt$ et $vt$ par $r_1$ sont les mêmes que les traces de leurs images par $r_2$. Si $r=r_1$, la trace correspondant au lacet faisant le tour de $x=i$ est donnée par $2cos(\pi \theta_i)$. On en déduit le tableau \ref{traces}.
\begin{table}[htb] \caption{Traces pour $r_1$, modulo l'action de $\mathcal{MCG}(\mathbb{S}_4^2)$ \label{traces}}
\begin{center}
 \begin{tabular}{|c||c|c|c|c|c|c|c|}
\hline
 $M$ & $V$&$W$&$T$&$U$&$VW$&$WT$&$VT$\\
 \hline
 $trace(M)$&$-\sqrt{3}$&$-2$&$-2$&$\sqrt{3}$&$1$&$0$&$1$\\
 \hline
\end{tabular}
\end{center}
 \end{table}
 
D'après \cite[Theorem 2.9]{MR1731936}, le groupe $<V,W>$ est Zariski dense dans $\mathrm{PSL}_2(\C)$ puisque, si $tr(A)$ désigne la trace d'une matrice $A$, on a $$tr(V)^2+tr(W)^2+tr(VW)^2-tr(V)tr(W)tr(VW)=8-2 \sqrt{3}.$$
Un quadruplet qui donne ces traces est donné ci-dessous.\\
$V_0=\left[ \begin {array}{cc} -\sqrt {3}&-2-\sqrt {3}
\\ \noalign{\medskip}2-\sqrt {3}&0\end {array} \right] 
$,
$ W_0= \left[ \begin {array}{cc} -1&-1-\sqrt {3}\\ \noalign{\medskip}0&-1
\end {array} \right] 
$,

$T_0=\left[ \begin {array}{cc} -3&-2-2\,\sqrt {3}\\ \noalign{\medskip}
\sqrt {3}-1&1\end {array} \right] 
$,
$U_0=\left[ \begin {array}{cc} \sqrt {3}-1&1\\ \noalign{\medskip}-2+\sqrt 
{3}&1\end {array} \right] 
$.\\
L'action de $\mathcal{MCG}(\mathbb{S}_4^2)$ ne change pas l'image de la représentation.
Le groupe de monodromie de $\calR_{s_0}$ est donc $<U_0,V_0,W_0>$.
Nous réemploierons ces informations à la section \ref{monod}.
\subsubsection{Un feuilletage non tiré-en-arrière}
En vue d'appliquer la proposition \ref{cor}, on va montrer ce qui suit.
\begin{lemm}\label{lemme non pb}
 Le feuilletage de Riccati $\calR$ n'est pas birationnellement équivalent à un feuilletage de Riccati tiré-en-arrière d'un feuilletage de Riccati au dessus d'une courbe. 
\end{lemm}
\begin{proof}
Procédons par l'absurde.
Soit $\calC$ une courbe compacte lisse.
Soient $\calR^*$ un feuilletage de Riccati sur $\Pu\times (\Pu\times \Pu) \rightarrow \Pu\times (\Pu\times \Pu)$ birationnellement équivalent à $\calR$, $\phi : \Pu \times \Pu \dasharrow \calC$ une application rationnelle et $\calR^0$ un feuilletage de Riccati sur $\Pu \times \calC \rightarrow \calC$ tel que $(Id \times \phi)^*\calR^0=\calR^*$. 

On va étudier la restriction $y=\phi_s(x)$ de $\phi$ à la valeur $s$ du paramètre de la déformation. L'application $\phi_s$ ne peut être constante pour un $s$ générique, sous peine de ne pas avoir de monodromie pour $\calR_s$. C'est donc que nous avons, pour $s$ générique, un revêtement ramifié $\phi_s : \Pu  \rightarrow \calC$ et que $\calC=\Pu$. 

Soit $\rho^0$ (resp. $\rho_s$) la représentation de monodromie de $\calR^0$ (resp. $\calR^*_s$).
On a $\rho_s=\rho^0\circ (\phi_s)_*=(\phi_s)^*\rho^0$. 
Quitte à réduire $\rho^0$ et à remplacer $\rho_s$ par une réduction partielle, on peut supposer $\rho^0$ réduite (définition $\ref{reduction}$). Soit $\{(q_j)_j\}$ le support de $\rho^0$.
Associons à la représentation $\rho^0$ la liste $(n_j)_j$ des ordres des éléments $\rho^0(\alpha_j)$ correspondants aux lacets simples $\alpha_j$ autour des $q_j$. Le théorème suivant donne la condition sous laquelle cette liste définit une structure orbifolde hyperbolique sur $\Pu$.

\begin{theo}[Poincaré]
Soient $\Sigma$ une surface de Riemann, $(u_i)_{i\in I}$ une famille finie de points de $\Sigma$ distincts deux à deux et $(m_i)_{i\in I}$ avec, pour tout $i$, $m_i\in \N^*\cup\{\infty\}$.
Soit $A((m_i)_i)=2\pi \left(2g(\Sigma)-2 +\sum_{i\in I} \left(1-\frac{1}{m_i}\right)\right)$; où $\frac{1}{\infty}:=0$. Si $A>0$ alors il existe un sous-groupe discret $\Gamma$ de $\mathrm{PSL}_2(\R)$ tel que $\Hhyp/ \Gamma \stackrel{hol.}{\simeq} \Sigma$ et tel que la métrique  sur $\Hhyp$ donne une métrique singulière exactement aux $(u_i)$ avec les angles $(\frac{2\pi}{m_i})$ autour de ces points et l'aire de $\Sigma$ pour cette métrique soit donnée par $A$.
\end{theo}

Comme $\rho_s=(\phi_s)^*\rho^0$, il y'a, parmi les $n_j$,  $n_{j_0}=6k$, $k\in \N^*$ et $n_{j_1}=\infty$. Comme $\rho^0$ est irréductible (vu que $\rho_s$ l'est, cf section \ref{monodlin}) il y'a un troisième $n_j$ non trivial : $n_{j_2}\geq 2$.
Ainsi $$A_1:=A((n_j)_j)\geq 2\pi \left(-2+\left (1-\frac{1}{\infty}\right)+\left(1-\frac{1}{6}\right)+\left(1-\frac{1}{2}\right)\right)\geq \frac{2\pi}{3},$$
et le théorème ci-dessus nous fournit une métrique singulière $\mu$ sur $\Pu$, d'aire totale $A_1$. L'aire $A_2$ de $\Pu$ pour $\phi_s^*\mu$ est $d A_1$, où $d$ est le degré de $\phi_s$. Les points singuliers de $\phi_s^*\mu$ sont de deux types : soit ils sont
dans les fibres des points singuliers de $\mu$ pour $\phi_s$, soit des points où $\phi_s$ n'est pas étale au dessus d'un point où $\mu$ n'a pas de singularité.

 Comme $\rho_s=(\phi_s)^*\rho^0$, il n'y a que quatre points du premier type dont les angles respectifs ne soient pas de la forme $2\pi\ell$, $\ell\in \N^*$. Leurs angles sont de la forme $(2a\pi/6, 2b\pi/6,0,0)$, où $a,b\in \N^*$. Le revêtement $\phi_s$ se déforme, vu que le birapport $t(s)$ des pôles de $\calR_s$ est non constant, et il y'a donc au moins un point du second type, pour $s$ général. L'angle autour d'un point du second type pour $\phi_s^*\mu$ est $2\pi r$ où $r$ est son indice de ramification.

De plus, comme on peut le voir à l'aide d'une triangulation géodésique adaptée,\\ $A_2=2\pi \left(-2+\left (1-\frac{1}{\infty}\right)+\left (1-\frac{1}{\infty}\right)+ \left( 1-\frac{a}{6} \right)+ \left(1-\frac{b}{6}\right)+\sum_i(1-r_i)\right)$ où, pour tout $i$, $r_i\in \N^*$ correspond à un point singulier d'angle $2\pi r_i$.
On a ainsi $d A_1=A_2 \leq 2\pi(-2+2(1-0)+2(1-\frac{1}{6})+(1-2))=\frac{4\pi}{3}$. Comme $\frac{2\pi}{3}\leq A_1$, on en tire  $d\frac{2\pi}{3}\leq\frac{4\pi}{3}$ ou encore $d \leq 2$.

La famille $(\phi_s)$ est donc une famille de revêtements de degré $2$ de $\Pu$ par lui même, c'est donc une déformation de $x\mapsto x^2$, consistant à faire bouger les points de ramification par rapport aux pôles de $\calR^0$. 
Vues les contraintes de monodromie, on a nécessairement $(n_j)_j=(2,6,\infty)$ et un des points de ramification est fixe au dessus du point associé à $n_j=2$, tandis que l'autre est mobile.

 La représentation de monodromie de $\calR^0$ est ainsi donné par $A,B\in \mathrm{PSL_2}$ avec $A$ parabolique, $B$ d'ordre $6$ et $AB$ d'ordre $2$; une étude élémentaire permet alors de voir que $\rho_s$ est donnée par 
le quadruplet $(A,B,A,B)$ modulo l'action de $\mathcal{MCG}(\mathbb{S}_4^2)$. On peut alors voir, par calcul, que l'orbite de ce quadruplet sous $\mathcal{MCG}_{pure}(\mathbb{S}_4^2)$ est de taille $2$, tandis que celle de $(V_0,W_0,T_0,U_0)$ est de taille $6$, ce qui donne une contradiction.
C'est une façon algébrique d'utiliser que la solution de Painlevé VI construite à partir du revêtement double ci-dessus n'a pas le même degré que la solution qui nous intéresse.
Le lemme est démontré.
 \end{proof}{}

\begin{rema}
\'Evidemment ce type de méthode peut s'appliquer à d'autres déformations isomonodromiques, voir \cite{arXiv:1201.1499}.
\end{rema}

\subsection{Quotient}\label{quotient}

On a trouvé un groupe d'automorphismes d'ordre $4$ pour le feuilletage $\calR$ introduit à la section \ref{non pb}.
À l'aide du calcul formel, on a déterminé une équation pour le feuilletage de Riccati quotient $\hat{\calR}$.

Ce dernier est défini par la forme $-dz+\hat{\alpha}+\hat{\beta} z+\hat{\gamma} z^2$ sur $\Pu \times (\Pu \times \Pu)$ avec $\hat{\alpha}$, $\hat{\beta}$ et $\hat{\gamma}$ donnés ci-dessous où on considère $x, y$ et $z$ comme des coordonnées affines sur $\Pu$,

\footnotesize 
 $\hat{\alpha}=-5\,{\frac {x \left( 3\,y+1 \right)  \left( -y+3\,x \right) dx }{ \left( -y+27\,{x}^{2}-6\,x \right)  \left( -2\,{y}^{2}-9\,
y+30\,yx+9\,{x}^{2} \right) }}-{\frac {5}{12}}\,{\frac { \left( 18\,{y
}^{2}x-10\,{y}^{2}+162\,y{x}^{2}-45\,yx+81\,{x}^{3}-18\,{x}^{2}
 \right) x dy}{ \left( -y+27\,{x}^{2}-6\,x \right) y
 \left( -2\,{y}^{2}-9\,y+30\,yx+9\,{x}^{2} \right) }}$,
 
$\hat{\beta}=\,{\frac { \left( -75\,{y}^{3}x+12\,{y}^{3}-45\,{y}^{2}{x}^{2}+54\,
{y}^{2}-40\,{y}^{2}x-1212\,y{x}^{2}+810\,y{x}^{3}+265\,yx+15\,{x}^{2}-
216\,{x}^{3} \right) dx}{ 6\left( -y+27\,{x}^{2}-6\,x
 \right) x \left( -2\,{y}^{2}-9\,y+30\,yx+9\,{x}^{2} \right) }}\\
+{
}\,{\frac { \left( 1350\,{y}^{4}x-966\,{y}^{4}+17010\,{y}
^{3}{x}^{2}-4707\,{y}^{3}x-256\,{y}^{3}-17064\,{y}^{2}{x}^{2}+2466\,{y
}^{2}x\right) dy }{ 72\left( -y+27\,{x}^{2}-6\,x \right) y
 \left( -2\,{y}^{2}-9\,y+30\,yx+9\,{x}^{2} \right)  \left( 3\,y+1
 \right) }}\\
 +\frac{\left(49815\,{y}^{2}{x}^{3}-428\,{y}^{2}-1845\,yx-30780\,y{x}^{3}+
14508\,y{x}^{2}+21870\,y{x}^{4}-5832\,{x}^{4}+1701\,{x}^{3}-90\,{x}^{2
} \right)dy}{ 72\left( -y+27\,{x}^{2}-6\,x \right) y
 \left( -2\,{y}^{2}-9\,y+30\,yx+9\,{x}^{2} \right)  \left( 3\,y+1
 \right) }$
 et \\
 $\hat{\gamma}=\,{\frac { \left( 226800\,{y}^{2}{x}^{2}-8325\,{y}^{3}
x-67665\,{y}^{2}x+33\,yx+2376\,y{x}^{2}-1080\,{x}^{2}+2593\,{y}^{2}+
565\,y-75\,x+1875\,{y}^{4}-5033\,{y}^{3} \right) dx }{720
 \left( -y+27\,{x}^{2}-6\,x \right) x \left( -2\,{y}^{2}-9\,y+30\,yx+9
\,{x}^{2} \right) }}\\
-\,{\frac { \left( -7950\,{y}^{5}
+33750\,{y}^{5}x+104906\,{y}^{4}-374769\,{y}^{4}x-36450\,{y}^{4}{x}^{2
}-6169500\,{y}^{3}{x}^{2}+776799\,{y}^{3}x+1322\,{y}^{3} \right) dy}{1864
 \left( -y+27\,{x}^{2}-6\,x \right) x \left( -2\,{y}^{2}-9\,y+30\,yx+9
\,{x}^{2} \right) y \left( 3\,y+1 \right) }}\\
-\frac{\left(11524275\,{y}
^{3}{x}^{3}-2426517\,{y}^{2}{x}^{3}+7654500\,{y}^{2}{x}^{4}+72945\,{y}
^{2}x-71244\,{y}^{2}{x}^{2}-4030\,{y}^{2}-951831\,y{x}^{3}+224028\,y{x
}^{2}-17325\,yx\right)dy}{1864
 \left( -y+27\,{x}^{2}-6\,x \right) x \left( -2\,{y}^{2}-9\,y+30\,yx+9\,{x}^{2} \right) y \left( 3\,y+1 \right) }\\
 -\frac{\left(1968300\,y{x}^{5}+801900\,y{x}^{4}-145800\,{x}^{4}-450
\,{x}^{2}+14985\,{x}^{3}+393660\,{x}^{5}\right)dy}{1864
 \left( -y+27\,{x}^{2}-6\,x \right) x \left( -2\,{y}^{2}-9\,y+30\,yx+9\,{x}^{2} \right) y \left( 3\,y+1 \right) }$.
 \normalsize
 
 Pour montrer que $\hat{\calR}$ est le quotient annoncé, on va considérer le tiré en arrière $\calR^1$ de $\hat{\calR}$ par le revêtement $rev:=(Id \times Id \times r)$ où $r :\Pu \rightarrow \Pu$ est défini par $r(s)={\frac {4{s}^{2}}{ \left( {s}^{2}-3 \right) ^{2}}}$.

Dans cette situation $\hat{\calR}$ est le quotient de $\calR^1$ par le groupe d'automorphisme $<s\mapsto -s,s\mapsto 3/s>$ du revêtement galoisien $rev$.
Il suffit alors de montrer ce qui suit.
\begin{lemm}\label{lemmquotient}
Pour $s\in \Pu$ général, il existe $\psi_s \in Aut(\Pu)$ tel que $(Id \times\psi_s)^*\calR^1_s$ soit birationnellement équivalent à $\calR_s$, disons via $\phi_s(x)$. 
En particulier, les réductions des représentations de monodromie  de $\calR_s$ et $\hat{\calR}_{r(s)}$ sont les mêmes. 
De plus, $(x,s)\mapsto(\phi_s(x),\psi_s)\in \mathrm{GL}_2(\C)^2$ est une application rationnelle sur $\Pu\times \Pu$.
\begin{proof}
Par calcul formel, l'automorphisme $\psi_s$ est donné par le changement de variable $x={\frac {6\,{s}^{2} \left( s-1 \right) ^{3}X-2\,s \left( s-3 \right) 
^{3}}{27\, \left( s-1 \right) ^{3} \left( {s}^{2}-3 \right) X+3\,s
 \left( s-3 \right) ^{3} \left( {s}^{2}-3 \right) }}$; 
 
 l'équivalence birationnelle  de $(Id \times \psi_s)^*\calR^1_s$ avec $\calR_s$ est alors donnée par
 $z=A(X,s)\cdot Z$ pour
$A(X,s)=\left [\begin{smallmatrix}a&b\\c&d\end{smallmatrix} \right ]$  \small
où
 $a=g h w, b=v g f, c=h u$ et $d=f e$; avec

$f={s}^{4}-8\,{s}^{3}-6\,{s}^{2}-24\,s+9;$

$g=40\,s \left( {s}^{2}-3 \right)  \left( {s}^{2}+3 \right)  \left( 3\,X{
s}^{4}-{s}^{3}-9\,X{s}^{3}+9\,X{s}^{2}+9\,{s}^{2}-3\,Xs-27\,s+27
 \right);$
 
$h={s}^{7}+11\,{s}^{6}-27\,{s}^{5}+75\,{s}^{4}+135\,{s}^{3}+405\,{s}^{2}+
675\,s+405;$

$v=X{s}^{7}+5\,{s}^{6}+33\,{s}^{5}X-30\,{s}^{5}+15\,{s}^{4}+140\,{s}^{3}-
125\,X{s}^{3}-45\,{s}^{2}+315\,Xs-270\,s-135;$

$w=-{s}^{4}+5\,{s}^{3}-3\,X{s}^{3}-3\,X{s}^{2}-3\,{s}^{2}+15\,Xs-9\,s-9\,
X;$
$e=-688905\,s-688905\,{s}^{2}+164025\,X+1980315\,{s}^{7}-164025\,Xs-
6748110\,X{s}^{4}-3268962\,X{s}^{7}+7416630\,{s}^{5}X-2262330\,X{s}^{3
}+2310930\,X{s}^{2}+2489535\,{s}^{4}+2897775\,{s}^{3}-1357965\,{s}^{6}
-3211245\,{s}^{5}-375\,{s}^{15}+25\,{s}^{16}-45\,{s}^{16}X+3861\,{s}^{
13}{X}^{2}+5\,{s}^{17}X-243\,{s}^{15}{X}^{2}-1071\,{s}^{14}{X}^{2}+298
\,{s}^{15}X-1602\,{s}^{14}X+81\,{s}^{16}{X}^{2}-54018\,{s}^{12}X+
172402\,{s}^{11}X-360810\,{s}^{10}X+1046682\,{s}^{6}X+504516\,X{s}^{9}
+546372\,X{s}^{8}+9914\,{s}^{13}X+31725\,{s}^{12}{X}^{2}-105687\,{s}^{
11}{X}^{2}+185229\,{s}^{10}{X}^{2}+242595\,{s}^{6}{X}^{2}-270351\,{s}^
{9}{X}^{2}-54621\,{s}^{8}{X}^{2}+789687\,{s}^{7}{X}^{2}-3042225\,{s}^{
5}{X}^{2}+2139615\,{X}^{2}{s}^{4}+2465235\,{X}^{2}{s}^{3}-3575745\,{X}
^{2}{s}^{2}+1191915\,{X}^{2}s+2225\,{s}^{14}-177525\,{s}^{8}+211365\,{
s}^{10}+12765\,{s}^{12}-6575\,{s}^{13}-46675\,{s}^{11}-432795\,{s}^{9}$;

$u=-45927\,{s}^{2}+91044\,{s}^{7}+69984\,Xs-476928\,X{s}^{4}+387936\,{s}^
{5}X+326592\,X{s}^{3}-186624\,X{s}^{2}+157464\,{s}^{4}+30618\,{s}^{3}-
53055\,{s}^{6}-121986\,{s}^{5}+768\,{s}^{12}X-4032\,{s}^{11}X+17664\,{
s}^{10}X-138240\,{s}^{6}X-43104\,X{s}^{9}+46080\,X{s}^{8}-96\,{s}^{13}
X+189\,{s}^{12}{X}^{2}-378\,{s}^{11}{X}^{2}-5832\,{s}^{10}{X}^{2}+
128088\,{s}^{6}{X}^{2}+13554\,{s}^{9}{X}^{2}+17685\,{s}^{8}{X}^{2}-
91044\,{s}^{7}{X}^{2}-101628\,{s}^{5}{X}^{2}+83511\,{X}^{2}{s}^{4}-
101250\,{X}^{2}{s}^{3}+97200\,{X}^{2}{s}^{2}-51030\,{X}^{2}s-5\,{s}^{
14}+10935\,{X}^{2}-42696\,{s}^{8}-3093\,{s}^{10}-400\,{s}^{12}+70\,{s}
^{13}+1250\,{s}^{11}+11292\,{s}^{9}$.

\normalsize

On donne en ligne une feuille de calcul qui permet de vérifier ces assertions.

\end{proof}
\end{lemm}
Dans la suite, on va considérer l'image $\tilde{\calR}$ de $\hat{\calR}$ par 
 $$\begin{array}{rcl} \Pu\times(\Pu\times \Pu)&\dasharrow &   \Pu\times \Pd \\
                         ([z:1], ([x:1],[y:1]))    &\longmapsto     &([z:1],[x:y:1]).
                      \end{array}$$

\subsection{Dimension de Kodaira numérique}

On choisit de s'intéresser au feuilletage $\F$ sur $\Pd$ dont la structure transverse est donnée par notre feuilletage de Riccati $\tilde{\calR}$ et la section $z=0$.
Ce feuilletage est ainsi donné, dans la carte affine $(x,y)$ par la forme 
$$\omega=-12y(1+3y)(3x-y)dx+\left [(10-18x)y^2-9x(18x-5)y-9x^2(9x-2) \right ]dy$$


Nous souhaitons décider si ce feuilletage est, à transformation birationnelle près, un feuilletage modulaire au sens de la définition \ref{modulaire}. Pour ce faire, on emploie la proposition \ref{cor}.

On va calculer la dimension de Kodaira numérique de $\F$, pour cela on désingularisera $\F$ et donnera le diviseur canonique $K_{\Ft}$ du feuilletage désingularisé $\Ft$.
 Comme l'étude de la section \ref{courbesinv} le montrera, $\Ft$ possède un nombre fini de courbes rationnelles invariantes,  ainsi $K_{\Ft}$ est pseudo-effectif, d'après Miyaoka; cf \cite[chapter 7]{MR2114696}. On pourra donc calculer la décomposition de Zariski $K_{\Ft}=P+N$.
 
D'après  un théorème de McQuillan (voir \cite[Theorem 1 p~106]{MR2114696}), pourvu que $\Ft$ soit relativement minimal, le support de $N$ est bien connu : c'est la réunion des supports des $\Ft$-chaînes maximales.
\newpage
Les $\Ft$-chaînes sont les courbes $C$ de composantes irréductibles $(C_i)_{i=1\ldots r}$ telles que, pour tous $i,j \in \{1,\ldots,r\}$, on ait
\begin{itemize}
\item $C_i$ est une courbe rationnelle  lisse invariante de $\Ft$,
 \item $C_i.C_j=1$ si $\vert i-j \vert=1$,
 \item $C_i.C_j=0$ si $\vert i-j \vert>1$,
 \item $C_i.C_i<-1$,
 \item $C_1$ contient une seule singularité de $\Ft$ et, pour $i \in  \{2,\ldots,r\}$, $C_i$ contient exactement deux singularités de $\Ft$.
 \end{itemize}
 
Pour un feuilletage réduit, la condition de relative minimalité est une condition qui porte sur les courbes rationnelles invariantes de $\Ft$ (cf \cite[chapter $5$]{MR2114696} ).
 
 On voit que la compréhension des courbes rationnelles invariantes de $\F$ est un préliminaire à la bonne réalisation de notre objectif. 
 
On sait que certaines composantes du lieu polaire de $\tilde{\calR}$ peuvent donner des courbes invariantes pour le feuilletage, ce qui se vérifie pour les composantes suivantes :
 \begin{equation}  \label{equ} \left \lbrace \begin{array}{lr} 
   \ell_1:&y=0 \mbox{ ;}\\
\ell_2:&y+\frac{1}{3}=0 \mbox{ ;}\\
 \ell_\infty:& \mbox{ la droite à l'infini ;}\\
  R:&-\frac{1}{27}y+x^{2}-\frac{2}{9}x=0 \mbox{ et}\\
 V:&-\frac{2}{9}y^{2}-y+\frac{10}{3}xy+{x}^{2}=0.
 \end{array} \right.
 \end{equation}

Chacune de ces composantes est une courbe rationnelle lisse. 

On calcule à l'aide de Maple l'ensemble des points singuliers $\Sigma=Sing(\F)$ du feuilletage $\F$ sur $\Pd$ et, pour tout $p\in \Sigma$, on calcule la partie linéaire $L(p)$ d'un champs de vecteurs à zéros isolés définissant $\F$ au voisinage de $p$.
Les singularités non-réduites sont les éléments $p$ de $\Sigma$ tels que $L(p) = 0$ ou  bien tels que le rapport $ \lambda =\lambda(p)$ des valeurs propres  de $L(p)$ soit un rationnel strictement positif.
Ce sont ces singularités qu'il faut éclater pour désingulariser le feuilletage $\F$, cf  \cite[chapter $1$]{MR2114696} .

On donne la liste des éléments $p$ de $\Sigma$ et leurs propriétés dans le tableau \ref{sing}.
On donne un dessin de la configuration de nos courbes dans la figure \ref{courbes}. On y place aussi les points singuliers qui, à l'exception du point $T$, correspondent à des intersections de nos composantes de pôle.

\begin{table}[htb]
\caption{ \label{sing} Description des points singuliers}
 \center{\begin{tabular}[t]{|c||c|c|c|}
 \hline
 $p$ &$ \lambda$ & élément de &coordonnées de $p$ \\
 \hline
 $A$  & $1$ &$\ell_1$, $\ell_2$ et $\ell_{\infty}$ & $(u,v)=(0,0)$\\
 \hline
 $B$ & $2 $& $V$ et $\ell_2$ &$ (x,y)=(5/9,-1/3)$\\
  \hline
$ C$ &$ 2$ & $R$ et $\ell_2$ &$(x,y)=(1/9,-1/3)$\\
  \hline
 $D$ & $-2$ & $R$ et $\ell_1$ &$ (x,y)= (2/9,0)$\\
  \hline
 $E,F$  &$\lambda <0$& $V$ et $\ell_{\infty}$   & $(u,v)=(15/2 \pm 9/2\,\sqrt {3},0)$ \\
  \hline
$ G$ &$ 2$ &  $R$ et $\ell_{\infty} $&$ (s,t)=(0,0)$\\
  \hline
$ H$ & partie linéaire nulle & $V$, $R$ et $\ell_1$ & $(x,y)=(0,0)$ \\
\hline
$I$ & 3& $V$ et $R$ & $(x,y)=(1/3,1) $\\
  \hline
 $T $ &$ -4$ & $\ell_2$ &$ (x,y)=(2/9,-1/3)$\\
  \hline
 \end{tabular}} \\
 \center{\textit{On utilise sur $\Pd$ les cartes $[x:y:1] \mapsto (x,y)$,  $[1:u:v] \mapsto (u,v)$ et $[s:1:t]\mapsto (s,t)$.}}
\end{table}

\begin{figure}[htb]
\caption{\label{courbes} configuration de courbes et positions des points}
\center{
\includegraphics{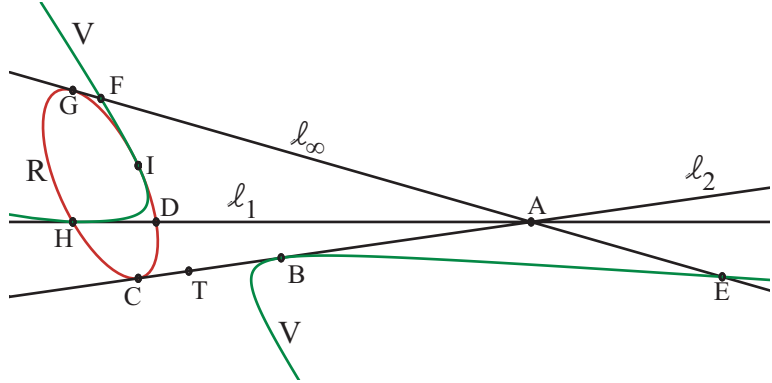}}
\end{figure}

\subsubsection{Désingularisation}\label{section desingularisation}
Par inspection de  la figure \ref{courbes}, on voit que si $p$ est élément de $ \{ A, B, C, G, I\}$ alors le feuilletage $\F$  a une singularité dicritique en $p$. Comme  $\lambda=\lambda(p) \in \N^*$, 
en appliquant le théorème de forme normale de Poincaré-Dulac on obtient que, dans de bonnes coordonnées $(z,w)$ au voisinage de $p$, le feuilletage est donné par le champ $z\partial_z+(\lambda w+\varepsilon z^{\lambda})\partial_w$, avec $\varepsilon=0$ ou $\varepsilon=1$.

 De plus, le fait que la singularité soit dicritique impose $\varepsilon=0$. À l'aide des formes normales, on peut alors voir que la singularité en $p$ se résout par une suite de $\lambda$ éclatements de diviseur exceptionnel une chaîne de courbes rationnelles $Ch(p)$, dont seulement 
la dernière composante est transverse au feuilletage et d'auto-intersection $-1$, tandis que les autres sont d'auto-intersection $-2$.

pour $p=B,C,G$, $\lambda=2$ on appelle $D_p$ la première composante de $Ch(p)$ et $E_p$ la dernière;
pour $p=A$, il n'y a qu'une composante à $Ch(p)$, qu'on nomme $D_A$ et
pour $p=I$, il y a trois composantes, la première est appelée $F_I$, la deuxième est appelée  $D_I$ et la dernière est appelée $E_I$.

Soit $X\rightarrow \Pd$ la désingularisation de $A,B,C,G,I$ décrite ci-dessus.
La désingularisation de $H$ n'est pas aussi prévisible. Toutefois, on voit, en calculant, que si on éclate $X$ en $H$, le feuilletage induit sur l'éclaté a trois singularités sur le diviseur exceptionnel $D_H$ : une réduite en $H_1$ avec $\lambda(H_1)=-7/2$, une réduite en $H_2$ avec $\lambda(H_2)=-7/12$ et une dicritique  en $H_3$ avec $\lambda(H_3)=1$. Il suffit ensuite d'éclater encore une fois en $H_3$ pour désingulariser le feuilletage, le nouveau diviseur exceptionnel $E_H$ est alors transverse au feuilletage, tandis que la transformée stricte de $D_H$, qu'on note encore $D_H$, est invariante par le feuilletage.

On note $\tilde{\mathbb{P}}^2 \rightarrow \Pd$ la désingularisation complète ainsi obtenue et $\Ft$ le feuilletage désingularisé, on donne un dessin décrivant cette désingularisation dans la figure
\ref{dessinglobaldesing}, les singularités y sont représentées par des points. Dans la figure \ref{chaines et cycles}, on donne un dessin plus épuré qui  indique les intersections entre les courbes rationnelles invariantes par $\Ft$ que nous connaissons. 
Le diviseur canonique du feuilletage $\F$ est aisément calculé : $\F$ est de degré $3$, si $\delta$ est une droite générique de $\Pd$, alors $K_\F=(3-1)\delta=2\delta$ d'après  \cite[p~27]{MR2114696}. Dans la suite d'éclatements on peut suivre ce que devient le diviseur canonique du nouveau feuilletage, c'est l'objet du résultat suivant.
\begin{prop}
Soient $ \pi : \tilde{Y} \rightarrow Y$ l'éclatement d'une surface projective lisse $Y$ en un point $p$ et $E$ le diviseur exceptionnel. Soit $\mathcal{H}$ un feuilletage sur $Y$ et $\tilde{\mathcal{H}}$ le feuilletage induit sur $\tilde{Y}$. Soit $a(p)$ l'ordre d'annulation d'une $1$-forme holomorphe $\omega$, à zéros isolés, donnant le feuilletage $\mathcal{H}$ au voisinage de $p$.
\begin{itemize}
\item Si $E$ est invariant par $\mathcal{H}$ alors $K_{\tilde{\mathcal{H}}}=\pi^*(K_{\mathcal{H}})- (a(p)-1)E$.
\item Si $E$ n'est pas invariant par $\mathcal{H}$ alors $K_{\tilde{\mathcal{H}}}=\pi^*(K_{\mathcal{H}})- a(p)E$.

\end{itemize}
\begin{proof}
Voir \cite[Chapter $2$]{MR2114696}.
\end{proof}
\end{prop}

Signalons que si $\mathcal{C}$ est une courbe, $\pi^*(\mathcal{C})$ est la \textbf{transformée totale} de  $\mathcal{C}$ par l'éclatement $\pi$.
On utilisera le nom des diviseurs de $\Pd$ pour désigner leurs \textbf{transformées strictes} par $\tilde{\mathbb{P}}^2 \rightarrow \Pd$.
Par utilisations successives de cette proposition, on obtient $K_{\Ft}=2\,\delta-D_{{A}}-E_{{B}}-E_{{C}}-D_{{H}}-2\,E_{{H}}-E_{{G}}-E_{{I}}$.
\begin{figure}[htbp]
\caption{Chaînes et cycles \label{chaines et cycles}}
\center{\includegraphics[scale=0.8]{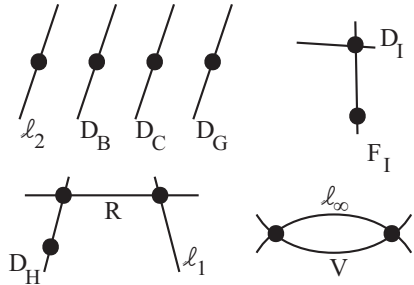}}
\end{figure}

\begin{figure}[htbp]
\caption{Désingularisation du feuilletage. \label{dessinglobaldesing}}
{\center{\includegraphics[angle=-90]{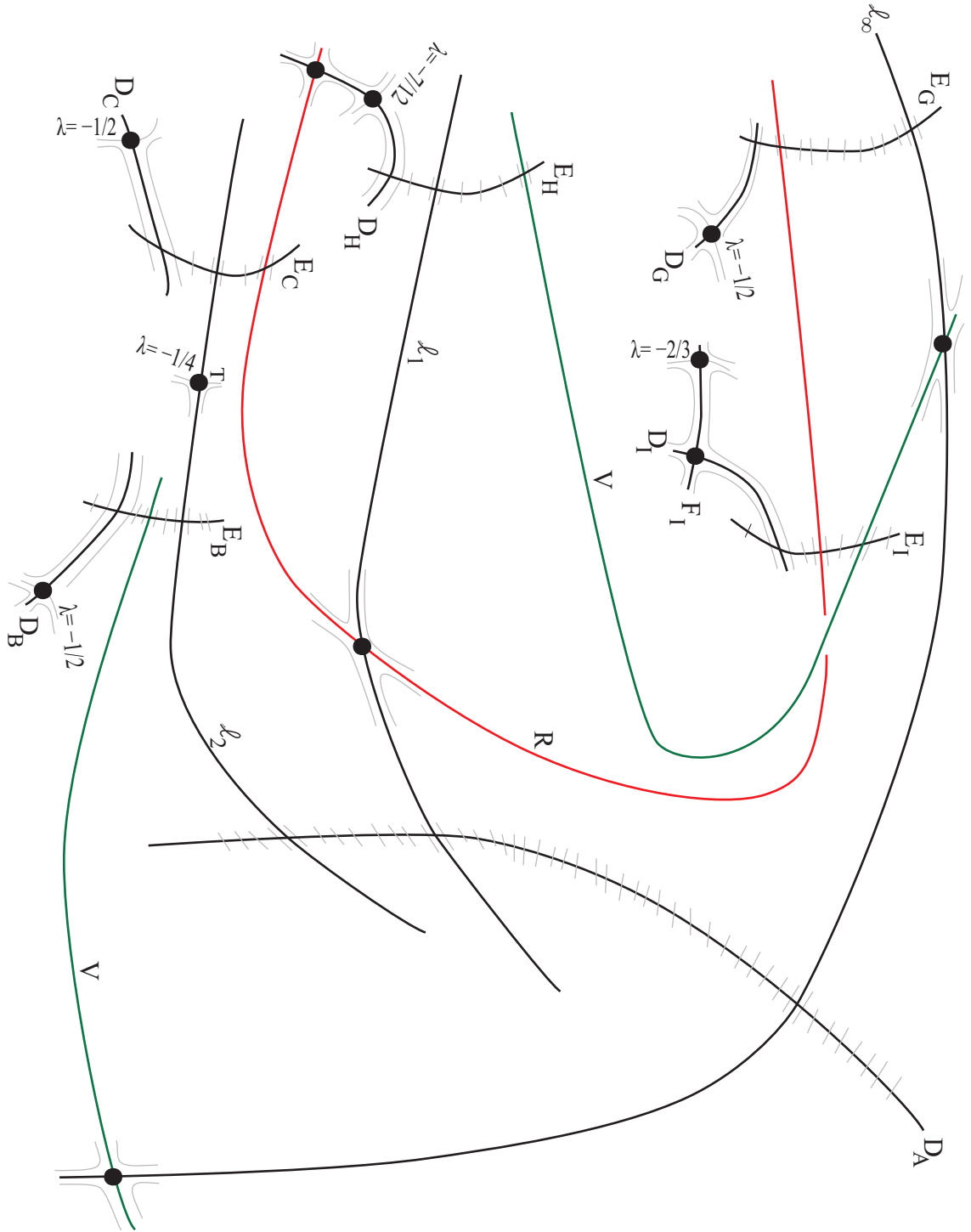}}}
\end{figure}
\clearpage

\subsubsection{\'Etude des courbes rationnelles invariantes} \label{courbesinv}
Montrons maintenant que $\tilde{\F}$ est relativement minimal et que toute composante d'une $\Ft$-chaîne est élément de 
$\mathcal{L}=\{\ell_1, \ell_2, \ell_{\infty}, V, R, D_B, D_C, D_G, D_H, D_I, F_I\}$.
Soit $\mathcal{C}_0$ une courbe rationnelle invariante de $\Ft$ qui n'appartient pas à $\mathcal{L}$.
On définit des inconnues pour le nombre d'intersections entre $\mathcal{C}_0$ et les diviseurs qui nous intéressent dans le tableau \ref{intcourbeinv}.

\begin{table}[htb] \caption{\label{intcourbeinv}Intersections de $\mathcal{C}_0$ avec nos diviseurs.}
\center{\scalebox{0.81}{\begin{tabular}{|c||c|c|c|c|c|c|c|c|c|c|c|c|c|c|c|c|c|c|}
\hline
$D$&$ R$ & $V$ & $\ell_{{1}}$ &$ \ell_{{2}}$ &$ \ell_{{\infty }}$ & $D_{{A}}$ & $D_{{B}}$ & $D_{{C}}$ &$ D_{{H}}$ & $D_{
{G}}$ & $D_{{I}}$ & $F_{{I}}$ & $E_{{H}}$ & $E_{{B}}$ & $E_{{C}}$ & $E_{{I}}$ & $E_{{G}}$ & $\delta$ \\
\hline
 $D.\mathcal{C}_0$&$0$&$0$&$0$&$\varepsilon_T$&$0$&$ n_A$& $\varepsilon_B$&$\varepsilon_C$&$\varepsilon_H$&$\varepsilon_G$&$0$&$\varepsilon_I$& $n_H$ & $n_B$ & $n_C$ & $n_I$ & $n_G$&$d$\\
\hline
\end{tabular}}}
\end{table}

Le degré de l'image de  $\mathcal{C}_0$ dans $\Pd$ est $d$. En appliquant le théorème de Bézout et en utilisant notre connaissance de la désingularisation du feuilletage, on obtient les équations suivantes qui lient les intersections lues dans $\Pd$ et celles lues dans $\tilde{\mathbb{P}}^2$.

 \begin{equation}  \label{equint} \left \lbrace \begin{array}{lr} 
  E_{ \ell_1}:&d=n_A+2n_H+ \varepsilon_H\\
E_{\ell_2}:&d=2n_B+ \varepsilon_B+2n_C+\varepsilon_C+ n_A+ \varepsilon_T\\
 E_{\ell_\infty}:&d=n_A+2n_G+\varepsilon_G\\
  E_R:&2d=n_H+\varepsilon_H+3n_I+\varepsilon_I+2n_G+\varepsilon_G+2 n_C+\varepsilon_C
\\
 E_V:&2d=2n_H+\varepsilon_H+3n_I+\varepsilon_I+2n_B+\varepsilon_B
 \end{array} \right.
 \end{equation}
D'autre part (cf \cite[Chapter $3$]{MR2114696}), l'auto-intersection de $\mathcal{C}_0$ peut être calculée en fonction des $\varepsilon_p$ à l'aide des indices de Camacho-Sad (qui correspondent aux $\lambda(p)$) : par la formule de Camacho-Sad,  on a  $\mathcal{C}_0^2=-\frac{1}{2}(\varepsilon_B+\varepsilon_C+\varepsilon_G)-\frac{7}{12}\varepsilon_H-\frac{2}{3}\varepsilon_I-\frac{1}{4}\varepsilon_T$. Comme les singularités sont réduites, on doit avoir, pour tout $p$, $\varepsilon_p \in \{0,1\}$. De plus, comme l'auto-intersection de $ \mathcal{C}_0$ est entière, les configurations envisageables pour les $\varepsilon_p$ se restreignent à celles données dans le tableau \ref{epsilon}  où l'on mentionne les valeurs de $\mathcal{C}_0^2$ qui y correspondraient.

\begin{table}[htb]
\caption{\label{epsilon} Auto-intersections envisageables pour $\mathcal{C}_0$.}
\center{\begin{tabular}{|c||c|c|c|c|c|c||c|}
\hline
n°&$\varepsilon_H$&$\varepsilon_I$&$\varepsilon_B$&$\varepsilon_C$&$\varepsilon_G$&$\varepsilon_T$&$\mathcal{C}_0^2$\\
\hline
\hline
$1$&$0$&$0$&$0$&$0$&$0$&$0$&$0$\\
\hline
$2$&$0$&$0$&$0$&$1$&$1$&$0$&$-1$\\
\hline
$3$&$0$&$0$&$1$&$0$&$1$&$0$&$-1$\\
\hline
$4$&$0$&$0$&$1$&$1$&$0$&$0$&$-1$\\
\hline
$5$&$1$&$1$&$0$&$0$&$1$&$1$&$-2$\\
\hline
$6$&$1$&$1$&$0$&$1$&$0$&$1$&$-2$\\
\hline
$7$&$1$&$1$&$1$&$0$&$0$&$1$&$-2$\\
\hline
$8$&$1$&$1$&$1$&$1$&$1$&$1$&$-3$\\
\hline

\end{tabular}}
\end{table}
Pour chacun de ces choix de $(\varepsilon_p)$, on peut résoudre le système linéaire $(\ref{equint})$. Seuls les choix n°$1$, $4$, $5$ et $8$ correspondent à des solutions entières pour ce système. L'ensemble des solutions $\sigma$ est alors un $\mathbb{Z}$-module de rang deux.
Toutefois on peut donner encore deux contraintes sur les inconnues du système $(\ref{equint})$. Premièrement, on peut calculer $\mathcal{C}_0^2$ en fonction de l'auto-intersection $d^2$ de sa projetée dans $\Pd$ et de la désingularisation $\tilde{\mathbb{P}}^2 \rightarrow \Pd$, en utilisant la proposition suivante.
\begin{prop}[voir \cite{MR507725}, p $187$ et p $476$]
Soient $ \pi : \tilde{Y} \rightarrow Y$ l'éclatement d'une surface $Y$ en un point $p$ et $E$ le diviseur exceptionnel.
\begin{enumerate}
\item Si $\mathcal{C}\subset Y$ est une courbe qui passe par $p$ avec multiplicité $m$ et  si $\tilde{\mathcal{C}}\subset \tilde{Y}$ est sa transformée stricte alors $\tilde{\mathcal{C}}^2=\mathcal{C}^2-m^2$.
\item De plus le diviseur canonique $K_{\tilde{Y}}$ de $\tilde{Y}$  est donné par  $K_{\tilde{Y}}=\pi^*(K_Y)+E$.
\end{enumerate}
\end{prop}
En appliquant successivement la première propriété, on obtient
\begin{equation} \label{autoint} \mathcal{C}_0^2=d^2-n_A^2-3n_I^2-2\varepsilon_I n_I-\varepsilon_I-\sum_{p=B,C,G,H}((n_p+\varepsilon_p)^2+n_p^2).
\end{equation}
En appliquant la seconde propriété et en sachant $K_{\Pd}=-3\delta$, on obtient : $K_{\tilde{\mathbb{P}}^2}=-3\,\delta+F_{{I}}+2\,D_{{I}}+D_{{C}}+D_{{H}}+D_{{G}}+D_{{A}}+D_{{B}}+2\,E_
{{B}}+2\,E_{{C}}+2\,E_{{G}}+2\,E_{{H}}+3\,E_{{I}}
$.
Deuxièmement, on utilise la formule du genre : $0=\frac{1}{2}\mathcal{C}_0\cdot(\mathcal{C}_0+K_{\tilde{\mathbb{P}}^2})+1$, pour obtenir 
\begin{equation} \label{genre}
0=1+\frac{1}{2}\,{d}^{2}-\frac{3}{2}\,d+\frac{1}{2}n_A+\frac{3}{2}n_I-\varepsilon_I n_I-\frac{1}{2}n_A^{2}-\frac{3}{2}n_I^{2}+\sum_{p=B,C,G,H}n_p(1-n_p-\varepsilon_p).
\end{equation}
Pour chacun des choix n°$1$, $4$, $5$ et $8$  en injectant les valeurs des $\varepsilon_p$, la valeur de $\mathcal{C}_0^2$ correspondante, ainsi qu'une paramétrisation de $\sigma$ par $(\lambda,\beta)\in\mathbb{Z}^2$ dans (\ref{autoint}), on arrive  à paramétrer $\lambda$ en fonction de $\beta$ puis à conclure en utilisant (\ref{genre}).
Les seules possibilités cohérentes sont alors celles données dans le tableau \ref{valpos}.
\begin{table}[htb] \caption{\label{valpos}}
\center{\begin{tabular}{|c|c|c|c|c|c|c|c|c|c|c|c|c|}
\hline
$n_H$&$\varepsilon_H$&$n_I$&$\varepsilon_I$&$n_B$&$\varepsilon_B$&$n_C$&$\varepsilon_C$&$n_G$&$\varepsilon_G$&$n_A$&$\varepsilon_T$&$d$\\
\hline

$1$&$ 1$& $2$& $1$&$ 1$&$ 0$& $0$& $0$& $1$&$ 1$& $3$& $1$&$ 6$\\
\hline
$3$& $1$& $5$& $1$&$ 2$&$ 1$&$ 0$& $1$& $3$& $1$& $7$& $1$& $14$\\
\hline

\end{tabular}}
\end{table}
Dans les deux cas, l'éventuelle courbe invariante $\mathcal{C}_0$ traverse plus de trois singularités, ce qui indique qu'elle n'est pas dans une $\Ft$-chaîne, de plus son auto-intersection est $-3$ ou $-2$, ce qui indique que le feuilletage $\Ft$ est relativement minimal et permet de conclure notre étude.

\subsubsection{Décomposition de Zariski}
La description de la désingularisation  que nous avons faite permet de donner, dans le tableau \ref{inter}, la matrice de la forme d'intersection sur $\tilde{\mathbb{P}}^2$ restreinte au sous-espace dont une famille génératrice est $$\mathcal{S}=(R,V,\ell_1,\ell_2,\ell_{\infty},D_A,D_B,D_C,D_H,D_G,D_I,F_I,E_H,E_B,E_C,E_I,E_G,\delta).$$

\begin{table}[b]
\caption{Matrice de la forme d'intersection pour la famille $\mathcal{S}$. \label{inter}}
\center{
$\left[ \begin{smallmatrix}
 -4&0&1&0&0&0&0&0&1&0&0&0&0&0&1&1&1&2
\\\noalign{\medskip}0&-3&0&0&2&0&0&0&0&0&0&0&1&1&0&1&0&2
\\\noalign{\medskip}1&0&-2&0&0&1&0&0&0&0&0&0&1&0&0&0&0&1
\\\noalign{\medskip}0&0&0&-4&0&1&0&0&0&0&0&0&0&1&1&0&0&1
\\\noalign{\medskip}0&2&0&0&-2&1&0&0&0&0&0&0&0&0&0&0&1&1
\\\noalign{\medskip}0&0&1&1&1&-1&0&0&0&0&0&0&0&0&0&0&0&0
\\\noalign{\medskip}0&0&0&0&0&0&-2&0&0&0&0&0&0&1&0&0&0&0
\\\noalign{\medskip}0&0&0&0&0&0&0&-2&0&0&0&0&0&0&1&0&0&0
\\\noalign{\medskip}1&0&0&0&0&0&0&0&-2&0&0&0&1&0&0&0&0&0
\\\noalign{\medskip}0&0&0&0&0&0&0&0&0&-2&0&0&0&0&0&0&1&0
\\\noalign{\medskip}0&0&0&0&0&0&0&0&0&0&-2&1&0&0&0&1&0&0
\\\noalign{\medskip}0&0&0&0&0&0&0&0&0&0&1&-2&0&0&0&0&0&0
\\\noalign{\medskip}0&1&1&0&0&0&0&0&1&0&0&0&-1&0&0&0&0&0
\\\noalign{\medskip}0&1&0&1&0&0&1&0&0&0&0&0&0&-1&0&0&0&0
\\\noalign{\medskip}1&0&0&1&0&0&0&1&0&0&0&0&0&0&-1&0&0&0
\\\noalign{\medskip}1&1&0&0&0&0&0&0&0&0&1&0&0&0&0&-1&0&0
\\\noalign{\medskip}1&0&0&0&1&0&0&0&0&1&0&0&0&0&0&0&-1&0
\\\noalign{\medskip}2&2&1&1&1&0&0&0&0&0&0&0&0&0&0&0&0&1\end{smallmatrix}
 \right] $}
\end{table}

D'après la section \ref{courbesinv}, $\Ft$ est relativement minimal et seuls des éléments de $\mathcal{L}$ entrent dans le support de la partie négative $N$ de la décomposition de Zariski de $K_{\Ft}$.
Ainsi, les $\Ft$-chaînes maximales sont les suivantes : 
$D_H+R+\ell_1$;
 $D_I+F_I$;
$\ell_2$;
 $D_B$;
 $D_C$ et
 $D_G$.

Le support de $N$ est donc $D_G\cup D_C \cup D_B  \cup \ell_2 \cup D_I \cup F_I \cup D_H \cup R \cup \ell_1$.
Ainsi $N=aD_G+ bD_C+ cD_B+d \ell_2+ e D_I + f F_I +g D_H + h R + i \ell_1$.
Chacune des composantes de ce support est orthogonale à $P \num K_{\Ft} -N$, ce qui donne un système d'équations linéaires simple qu'on résout :
$$N= \frac{1}{12}\,D_{{H}}+\frac{1}{6}\,R+{\frac{7}{12}}\,\ell_{{1}}+\frac{1}{2}\,D_{{G}}+\frac{1}{2}\,D_{{B}}+\frac{1}{2}\,D_{{C}}+\frac{1}{3}\,F_{{I}}+\frac{2}{3}\,D_{{I}}+\frac{1}{4}\,\ell_{{2}}.$$

On voit alors que $P \num K_{\Ft}-N$ n'est pas numériquement trivial puisque $(K_{\Ft}-N)\cdot\delta=\frac{5}{6}$, puis on constate que $P^2=0$.
Nous venons donc d'obtenir : $$\nu(\F)=1.$$

\begin{theo}\label{theomod}
À transformation birationnelle près, le feuilletage $\F$ est un feuilletage modulaire.
\begin{proof}
Il s'agit d'appliquer la proposition \ref{cor} à $\F$, vérifions en les hypothèses.
Une structure transversalement projective pour $\F$ est donnée par $\tilde{\calR}$. Par le lemme \ref{lemmquotient} et la section \ref{monodlin}, nous savons que le groupe de monodromie de $\tilde{\calR}$ est Zariski-dense donc non virtuellement abélien. Par la section \ref{quotient}, si $\tilde{\calR}$ était birationnellement tiré-en-arrière d'un feuilletage de Riccati au dessus d'une courbe, il en serait de même pour $\calR$ ; ce qui est exclu par le lemme \ref{lemme non pb}. Comme on vient de voir $\nu(\F)=1$, on peut conclure.
\end{proof}
\end{theo}

\section{Raffinements}\label{raffinements}
On cherche a identifier précisément le groupe $\Gamma$ qui apparaît dans la construction de la surface modulaire sous-jacente à $\F$.
Pour cela, on trouve le feuilletage dual $\G$ de $\F$ et sa structure transversalement projective, puis calcule les représentations de monodromie pour la structure transverse de $\F$ et celle de $\G$. Par le corollaire \ref{coro1} ci-dessous, le produit des deux représentations a pour image $\Gamma$.

\begin{lemm}\label{lemclean}
Soit $\F$ un feuilletage holomorphe singulier sur une surface projective lisse $X$.
Soit $D$ un diviseur sur $X$.
Soient $\Sigma=(P,\calR,\sigma)$ une structure transversalement projective pour $\F$ et $\Sigma_0=(P_0,\calR_0,\sigma_0)$ une structure transversalement projective pour $\F_{\vert X \setminus D}$.

Si la monodromie de $\calR$ n'est pas virtuellement abélienne et $\F_{\vert X \setminus D}$ ne possède pas d'intégrale première méromorphe non constante, alors $(P_0,\calR_0,\sigma_0)$ est biméromorphiquement équivalent à $(P,\calR,\sigma)_{\vert X \setminus D}$.
En particulier $\calR_{\vert X \setminus D}$ et $\calR_0$ ont même représentation de monodromie (après réduction).

\begin{proof}
Comme $X$ est projective, à transformation biméromorphe près, $\Sigma$ est le fibré trivial sur $X$ muni d'une équation de Riccati globale $\calR : dz=\alpha+\beta z+\gamma z^2$ et de la section triviale $z=0$.
D'autre part, d'après \cite[Remark 2.3.]{MR2337401}, $\Sigma_0$ peut être décrite (biméromorphiquement) par un recouvrement d'ouverts $(U_i)$ de $X \setminus D$, des équations de Riccati $(\calR_i)$ sur les $U_i$ et des transitions \textbf{méromorphes} entre ces feuilletages, qui respectent les sections $z=0$.

Par changements de trivialisations méromorphes au dessus des $U_i$, on peut supposer $\calR_i: dz=\alpha +\beta z+\gamma_i z^2$ (voir \cite[eq $(2.5)$ p $728$]{MR2337401}). Alors les transitions entre les $\calR_i$ sont triviales : les $1$-formes $\gamma_i$ se recollent en une $1$-forme $\tilde{\gamma}$ méromorphe sur $X\setminus D$.
Les triplets $(\alpha,\beta,\gamma)$ et $(\alpha,\beta,\tilde{\gamma})$ donnent alors deux structures transversalement projectives sur le fibré trivial pour $\F_{\vert X\setminus D}$
et on est dans le cadre de la preuve du lemme \ref{SLP}.
Ainsi, si $\gamma\neq \tilde{\gamma}$, soit il existe une une intégrale première méromorphe pour $\F_{\vert X\setminus D}$, soit la monodromie de $\calR_{\vert X\setminus D}$ est virtuellement abélienne. Ces deux situations sont exclues par hypothèse.
\end{proof}
\end{lemm} 
\begin{coro}\label{coro1}
Soit $Z_{\Gamma}$ une surface modulaire, $\F$ un des feuilletages modulaires sur $Z_{\Gamma}$ et $D$ la réunion de ses courbes algébriques invariantes. Si $\F$ possède une structure transversalement projective $\Sigma$ sur $Z_{\Gamma}$ à monodromie non virtuellement abélienne, alors la monodromie de $\Sigma_{\vert Z_{\Gamma}\setminus D}$  est une des projections de la représentation tautologique de $Z_{\Gamma}\setminus D$.
\begin{proof}
Toutes les feuilles de $\F_{\vert X\setminus D}$ sont denses en vertu du résultat de quasi-minimalité de \cite{MR2142243} qui s'étend au cas où $\Hhyp^2/\Gamma$ est compact. En effet, les feuilles de $\F_{\vert X\setminus D}$ sont denses si la projection $\Gamma_1$ de $\Gamma$ est dense, ce que nous allons démontrer.

Par irréductibilité de $\Gamma$, grâce à \cite[Theorem $1$]{MR0145106}, on sait que cette projection ne peut être discrète. Comme $\Gamma$ est un réseau de $\mathrm{PSL}_2(\R)^2$, il est Zariski dense, donc sa projection $\Gamma_1$ aussi. Pour conclure, on utilise que tout sous groupe Zariski dense de $\mathrm{PSL}_2(\R)$ est dense ou discret.

Par cette propriété de densité, on obtient que les intégrales premières méromorphes de $\F_{\vert X\setminus D}$ sont constantes. Par le lemme ci-dessus, la monodromie de $\Sigma_{\vert Z_{\Gamma}\setminus D}$ est donc celle de la structure transverse du lemme \ref{lemmemonodmodulaire}.

\end{proof}
\end{coro}

\subsection{Feuilletage dual}
Le résultat suivant est connu, on l'utilise pour chercher le feuilletage dual de $\F$.
\begin{theo}\label{minimal} Soit $Z_{\Gamma}$ une surface modulaire.
\begin{enumerate} 

 \item \label{minimal1} Les feuilletages $\F_{\Gamma}$ et $\G_{\Gamma}$ sont minimaux au sens de \cite{MR2114696}.
 \item \label{minimal2} De plus, leur diviseur de tangence est réduit  et son support est la réunion des chaînes de Hirzebruch-Jung et des cycles de courbes rationnelles.

 \end{enumerate}
\end{theo}
\begin{proof}

D'après la discussion \cite[pp $134-135$]{MR2114696}, les courbes algébriques irréductibles invariantes par $\F_{\Gamma}$ sont seulement les composantes des chaînes et cycles cités plus haut et, sur les cycles, la tangence entre les feuilletages est simple et les singularités sont réduites.
Sur les chaînes de Hirzebruch-Jung, le fait que les singularités  soient réduites  et la tangence simple peut se voir par un calcul aisé à l'aide de \cite[\S $3$ pp $220-223$]{MR0367276} qui donne des descriptions locales explicites de $\Hhyp^2 \rightarrow X_{\Gamma}$  et de la désingularisation de $X_{\Gamma}$.   Ainsi \ref{minimal2}. est démontré.  

Pour \ref{minimal1}, il suffit de voir que $\F_{\Gamma}$ est relativement minimal.
En effet, d'après \cite[Theorem 1 p 75]{MR2114696} et sa preuve, si $\F_{\Gamma}$ est relativement minimal sans être minimal, alors il est birationnellement un feuilletage de Riccati, le feuilletage très spécial de Brunella ou une fibration rationnelle, mais cela est exclu par la classification birationnelle des feuilletages.
Les composantes des chaînes et des cycles sont d'auto-intersection inférieure à $-2$ et $\F_{\Gamma}$ est réduit, il est donc relativement minimal, ce qui donne \ref{minimal1}.
\end{proof}

Par le point $\ref{minimal1}$ du théorème \ref{minimal}, comme $\Ft$ est relativement minimimal, ce qu'on a obtenu dans la section $2$ est $(\tilde{\mathbb{P}}^2,\Ft)=(Z_{\Gamma},\F_{\Gamma})$, pour un certain $\Gamma$; c'est la définition de la minimalité. On se propose de déterminer le feuilletage $\G$ sur $\Pd$ qui correspond au feuilletage $\G_{\Gamma}$ sur $\tilde{\mathbb{P}}^2$. Par le point $\ref{minimal2}$ du théorème \ref{minimal}, les feuilletages $\F$ et $\G$  ont un diviseur de tangence réduit $Tang$, donné par le lieu polaire de la structure transverse de $\F$, c'est à dire de degré $7$. De plus, le diviseur de tangence de deux feuilletages sur $\Pd$ de degrés  $d$ et $d'$  est de degré $d+d'+1$. Comme $\F$ est de degré $3$ on obtient que $\G$ l'est aussi.
Par un calcul affine, on voit que l'ensemble des feuilletages de degré $3$ ayant $Tang -\ell_{\infty}$ parmi leurs courbes invariantes est un pinceau $(\F_t)_{t \in \Pu}$ dont seul un élément ne laisse pas invariant $\ell_{\infty}$. Ce pinceau est donné par $\omega_t=P_tdx+Q_t dy$ où \footnotesize $P_t=-12\, \left( 1+3\,y \right) y \left( -2\,ty-y-10\,t+36\,tx+3\,x
 \right) $ et $Q_t= \left( -81-162\,t \right) {x}^{3}+ \left( -162\,y+756\,ty+18+306\,t \right) {x}^{2}+ \left( -36\,{y}^{2}t-18\,{y}^{2}+45\,y-270\,ty-60\,t
 \right) x+10\,y \left( y+t \right)$.
\normalsize\\
On remarque que
 \begin{equation*}  \left \lbrace \begin{array}{l} 
   s_1=\left \{(x,y)=\left( \,{\frac {10t}{9(1+2\,t)}},\,{\frac {20t \left( 3\,t-1 \right) }{ 3\left( 1+2\,t \right) ^{2}}}\right) \right \}\subset R \ \mbox{et}\\
s_2=\left \{(x,y)=\left(\,{\frac {10t \left( -9+2\,t \right) }{3(44\,{t}^{2}-96\,t-9)}},\,{\frac {-100{t}^{2}}{44\,{t}^{2}-96\,t-9}}\right) \right \} \subset V
 \end{array} \right.
 \end{equation*} sont deux singularités de $\F_t$, cela va nous permettre de déterminer le $t$ tel que $\F_t=\G$. En effet, comme $\G_{\Gamma}$ n'a pas de singularité sur $R$ et $V$ en dehors des intersections avec les autres courbes rationnelles invariantes, le $t$ recherché doit être tel que $s_1(t)\in\{C,D,G,H,I\}$ et $s_2(t)\in \{B,E,F,H,I\}$. La seule valeur de $t$ qui satisfasse ces deux conditions est $t=3/4$, on a donc finalement $\G=\F_{3/4}$.
On a trouvé une involution holomorphe de $\tilde{\mathbb{P}}^2$ (voir  section \ref{invol})  qui  échange $\F$ et $\G$, ainsi une structure transversalement projective pour $\G$, unique d'après le lemme \ref{SLP}, peut être obtenue en tirant en arrière celle de $\F$ par $(Id \times \sigma)$.

\subsection{L'involution \label{invol}}

Cette involution est donnée dans la carte affine $(x,y)$ de $\Pd$ par : $$\sigma : (x,y) \mapsto \left({\frac {3\,y \left( 3\,y+13 \right) x-y \left( 7\,y+9 \right) }{
 \left( 135\,y+9 \right) x-3\,y \left( 3\,y+13 \right) }},y\right).$$ On voit que c'est une transformation de  de Joncquières.

On l'a découverte en étudiant les tangences entre le pinceau des droites issues de $A$ et les feuilletages $\F$ et $\G$.
C'est une transformation de jauge méromorphe du $\Pu$-fibré $(x,y)\mapsto y$ d'espace total l'éclatement de $\Pd$ en $A$, qui est holomorphe en dehors de $\{(y-1)y=0\}$.
Quand on la conjugue par la désingularisation de $\F$, on obtient une transformation holomorphe de $\tilde{\mathbb{P}}^2=\tilde{Y}_{\Gamma}$ qui échange $D_I$ et $F_I$ ainsi que $E_I$ et $\{y=1\}$.
De même, elle échange $\ell_1$ et $D_H$ en fixant globalement $E_H$ et en réalisant un automorphisme non-trivial de $R$.
Comme indiqué précédemment, cette involution échange $\F$ et $\G$. 

\subsection{Calcul de la monodromie des structures transverses} \label{monod}
On souhaite comprendre ici les représentations de monodromie de $\tilde{\calR}$ et $(Id \times \sigma)^*\tilde{\calR}$.
\subsubsection{Groupe fondamental}
Un préalable à notre calcul de monodromie est la compréhension du groupe fondamental du complémentaire d'une courbe dans une surface.
L'outil principal utilisé pour cela est le théorème topologique suivant, dû à Zariski et Van-Kampen :
\begin{theo}
Soit $p : E \rightarrow B$ un fibré localement trivial qui possède une section $s$, avec $E$ connexe par arcs. Soit $b \in B$ et $F_b$ sa fibre.

Le groupe fondamental de $E$ est donné par la suite exacte scindée suivante  $0\rightarrow \pi_1(F_{b},b)\stackrel{i_*}{\rightarrow} \pi_1(E,b)\stackrel{\pi_*}{\rightarrow}\pi_1(B,b) \rightarrow 0$. La section est donnée par $s_*$.
\end{theo} 
\begin{proof}
Voir \cite{shimada} ou \cite[Theorem 2.1 p 170]{MR2830090}
\end{proof}
Précisons l'action du facteur  $\pi_1(B,b)$ sur $\pi_1(F_{b},b)$ pour le produit semi-direct $\pi_1(B,b) \ltimes \pi_1(F_{b},b)$ induit par cette suite exacte. Soit $\gamma : [0,1] \rightarrow B $ un lacet partant de $b$, $\gamma^*E$ est un fibré localement trivial de base contractile, il est donc trivialisable : $\gamma^*E \cong [0,1] \times F_b$ et tout lacet $\tau_0$ de $F_b$ de point de base $b$ se déforme continument en $\tau_t$, un lacet de point de base $\gamma(t)$ dans la fibre $F_{\gamma(t)}$. On définit ainsi une action de $\gamma \in \pi_1(B,b)$ sur $\pi_1(F_{b},b)$ en posant $\tau_0\cdot\gamma=\tau_1$ ; c'est l'action qui intervient dans la structure de produit semi-direct mentionnée ci-haut.

 Dans le cas où $F_b$ est un disque épointé, l'action correspond à l'action d'une tresse, comme indiqué dans la figure \ref{tresses}.   
\begin{figure}[htb]
\caption{Action de la tresse $\sigma_2$ \label{tresses}}
\center{
\includegraphics[scale=0.5]{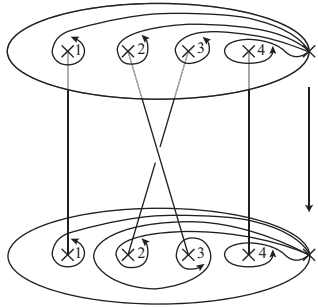}}
\end{figure}

 Soit $X \stackrel{\psi}{\rightarrow} \Pd$ l'éclatement de $\Pd$  en $A$ (comme dans la figure \ref{courbes}) et $E$ le diviseur exceptionnel. On va travailler avec $\tilde{\calR}'=\psi^*\tilde{\calR}$.
Remarquons que $X$ est muni d'une structure de $\Pu$-fibré : $X\stackrel{\pi}{\rightarrow}E$. Soit $P$ le support de la monodromie de $\tilde{\calR}'$, grâce au lemme \ref{lemmquotient} on voit que $E$ n'est pas une composante de $P$.
Les composantes de $P$ sont donc les transformées strictes des composantes du lieu polaire de $\tilde{\calR}$ : trois fibres de $\pi$, à savoir $\ell_1$, $\ell_2$ et $\ell_{\infty}$, et les deux coniques. Soit  $\ell_I$ la fibre de $I$ pour $\pi$. Soit $Q=P \cup \ell_I$, $X^{0}=X \setminus Q$ et $E^{0}=E\setminus Q$, on voit que $\pi_{\vert X^{0}}:X^{0} \rightarrow E^{0}$ est un fibré topologiquement localement trivial de fibre  $\mathbb{S}_4^2$, la sphère privée de quatre points. De plus, ce fibré admet une section, donnée par $E^{0}$; on a ainsi, par le théorème ci-dessus, une description du groupe fondamental de $X^{0}$ comme produit semi-direct :
$\pi_1(X^{0},b)=\pi_1(F_{b},b) \rtimes \pi_1(E^{0},b)$.

Reste à identifier de façon effective les deux facteurs et l'action.  Pour la base on choisit les générateurs $\alpha$, $\beta$, $\gamma$ comme sur la figure \ref{base}, où l'on rapporte les fibres à l'ordonnée $y$ de leur point d'intersection avec $\{x=0\}$.
On choisit pour $F_b$  la transformée stricte de $\{y=4\}$ par l'éclatement de $\Pd$ en $A$. On choisit alors comme générateurs pour $\pi_1(F_{b},b)$ les lacets $t,u,v,w$ comme sur la figure  \ref{fibre}, où l'on utilise l'abscisse $x$ des points comme coordonnée ; la seule relation entre ces lacets est $tuvw=1$. 

 \begin{figure}[htb]
\caption{\label{base} Chemins de la base $E^0$}
\center{
\includegraphics[scale=0.7]{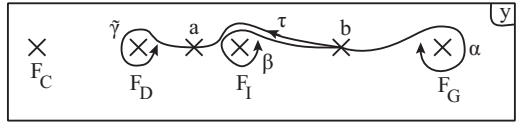}}
\end{figure}
\begin{figure}[htb]
\caption{\label{fibre} Lacets de la fibre $F_b$}
\center{
\includegraphics[scale=0.7]{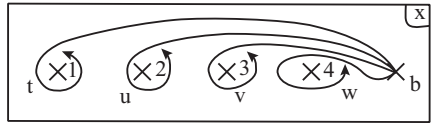}}
\end{figure}
 
On identifie les points $1$, $2$, $3$ et $4$  de $F_b$ avec les points $1$, $2$, $3$ et $4$ du disque épointé  quatre fois (figure \ref{tresses}), ce qui permet d'interpréter l'action de $\alpha$, $\beta$ et $\gamma$ à l'aide de l'action du groupe de tresses $\mathcal{B}_4$. Comme on le voit sur la figure \ref{tresses}, $\mathcal{B}_4$ agit -à droite par morphisme de groupe- sur $(t,u,v,w)$ de la manière suivante :
 \begin{equation}  \label{actionb4} \left \lbrace \begin{array}{l} 
   \  \sigma_1 : (t,u,v,w)\mapsto (tut^{-1},t,v,w)\\
   \ \sigma_2 : (t,u,v,w) \mapsto (t,uvu^{-1},u,w)\\
\  \sigma_3 : (t,u,v,w)\mapsto (t,u,vwv^{-1},v)
 \end{array} \right.
 \end{equation}
 Pour identifier les tresses associées à $\alpha$, $\beta$ et $\gamma$, on doit étudier l'évolution avec $\lambda \in [0,1]$ des positions relatives des éléments de $P\cap F_{y(\lambda)}$, pour différents chemins $y(\lambda)$ dans $E^0$. Notons que si $y(\lambda)$ est à valeurs dans $E^0\cap  ]0,+\infty[$, alors tout point de $P\cap F_{y(\lambda)}$ est réel ; l'étude se réduit alors à la lecture de la figure~\ref{courbes}. Les difficultés proviennent donc des situations où $y(\lambda)$ contourne ou entoure $F_D$, $F_I$ ou $F_G$ et se traitent localement.
 
  Par des changements de variables locaux réels de la forme $(\tilde{x},\tilde{y})=(\psi(x,y),\phi(y))$ avec $\phi$ croissante, on a les descriptions simples suivantes pour $P$. Au voisinage de $G$, $P : \tilde{y}(\tilde{y}+\tilde{x}^2)=0$; au voisinage de $I$,  $P: \tilde{x}(\tilde{x}+\tilde{y}^3)=0$ et au voisinage de $H$, $ P: \tilde{x}\tilde{y}(\tilde{y}-\tilde{x}^2)=0$.
 
De cette façon, on voit que $\alpha$ agit comme la tresse $\sigma_2$, que 
$\beta$ correspond à la tresse $\sigma_3^6$ et que l'action de $\gamma=\tau\tilde{\gamma}\tau^{-1}$ est celle de $\sigma_3^3$ suivie de celle de $\sigma_1\sigma_2\sigma_1$ sur la fibre $F_a$ de $a$, puis enfin suivie de l'action de $\sigma_3^{-3}$.

Puisque des générateurs pour les facteurs sont donnés par les familles $f=(\alpha,\beta,\gamma)$ et $g=(t,u,v,w)$, notre structure de produit semi-direct pour $\pi_1(X^{0},b)$ nous fournit la description par générateurs et relations suivante : $$\pi_1(X^{0},b)=<f \cup g \vert f_i^{-1}g_jf_i= g_j.f_i,tuvw=1>.$$ 
De manière explicite, l'ensemble des relations est le suivant.
 \begin{equation} \label{gen}
 tuvw=1
\end{equation}
 \begin{equation}  \label{rela} \left \lbrace \begin{array}{l} 
\ \alpha^{-1}t\alpha=t\\
\ \alpha^{-1}u\alpha=uvu^{-1}\\
\ \alpha^{-1}v\alpha=u \\
\ \alpha^{-1}w\alpha=w
 \end{array} \right.
 \end{equation}
  \begin{equation}  \label{relb} \left \lbrace \begin{array}{l} 
  \ \beta^{-1}t\beta=t\\
  \ \beta^{-1}u\beta=u\\
 \ \beta^{-1}v\beta=(vw)^{3}v(vw)^{-3}\\
\ \beta^{-1}w\beta=\left((vw)^2v\right)w\left((vw)^2v\right)^{-1} 
\end{array} \right.
 \end{equation}
   \begin{equation}  \label{relg} \left \lbrace \begin{array}{l} 
   
   \ \gamma^{-1}t\gamma=\left(tu(vw)^{-1}\right)w\left(tu(vw)^{-1}\right)^{-1} \\
   \ \gamma^{-1}u\gamma=tut^{-1} \\
\ \gamma^{-1}w\gamma=\left(t(wvw)^{-1}(vw)^2\right)t\left(t(wvw)^{-1}(vw)^2\right)^{-1} \\

\ \gamma^{-1}v\gamma=\left(t(wvw)^{-1}(vw)^2t(wvw)^{-1}\right)v\left(t(wvw)^{-1}(vw)^2t(wvw)^{-1}\right)^{-1} 

\end{array} \right.
 \end{equation}

Par la relation (\ref{gen}) et le fait que  $\pi_1(E^{0},b)$ agit par morphismes de groupe, on peut déduire une des quatre équations de (\ref{rela}), (\ref{relb}) ou (\ref{relg}) des trois autres. On peut donc simplifier la présentation en enlevant une équation de son choix dans chacune des familles de relations (\ref{rela}), (\ref{relb}) et (\ref{relg}), par exemple la plus longue, ce qui donne
l'ensemble de relations suivant, où l'on a aussi appliqué $tu=(vw)^{-1}$ dans la relation  $\gamma^{-1}t\gamma=\left(tu(vw)^{-1}\right)w\left(tu(vw)^{-1}\right)^{-1}$.
  \begin{equation}  \label{relf} \left \lbrace \begin{array}{l} 
  \ tuvw=1\\
  \ \alpha^{-1}t\alpha=t\\
  \ \alpha^{-1}v\alpha=u \\
\ \alpha^{-1}w\alpha=w\\
 \ \beta^{-1}t\beta=t\\
  \ \beta^{-1}u\beta=u\\
 \ \beta^{-1}v\beta=(vw)^{3}v(vw)^{-3}\\
   \ \gamma^{-1}t\gamma=(vw)^{-2}w(vw)^{2} \\
   \ \gamma^{-1}u\gamma=tut^{-1} \\
\ \gamma^{-1}w\gamma=\left(t(wvw)^{-1}(vw)^2\right)t\left(t(wvw)^{-1}(vw)^2\right)^{-1} \\

\end{array} \right.
 \end{equation}

D'après le théorème de Van Kampen sur le groupe fondamental de l'union de sous espaces topologiques, pour avoir une présentation du groupe fondamental de $X^1=X \setminus P$,  il suffit d'ajouter à cette présentation la relation $\beta=1$.
Ce groupe est donc donné par les générateurs $\{\alpha, \gamma , t, u, v, w \}$ et les relations :
  \begin{equation}  \label{relfu} \left \lbrace \begin{array}{l} 
  \ tuvw=1\\
  \ \alpha^{-1}t\alpha=t\\
  \ \alpha^{-1}v\alpha=u \\
\ \alpha^{-1}w\alpha=w\\
 \ v=(vw)^{3}v(vw)^{-3}\\
     \ \gamma^{-1}t\gamma=(vw)^{-2}w(vw)^{2} \\
   \ \gamma^{-1}u\gamma=tut^{-1} \\
\ \gamma^{-1}w\gamma=\left(t(wvw)^{-1}(vw)^2\right)t\left(t(wvw)^{-1}(vw)^2\right)^{-1} \\

\end{array} \right.
 \end{equation}

\subsubsection{Monodromie} \label{monod}
On va décrire la représentation de monodromie réduite $\rho : \pi_1(X^{1},b) \rightarrow \mathrm{PSL}_2(\C)$ associée à $\tilde{\calR}'$.
Il nous suffira de la comprendre en restriction à $\pi_1(F_b)$ et à $\pi_1(E \setminus P)$ pour la comprendre globalement. Remarquons que toute représentation de l'un de ces groupes vers $\mathrm{PSL}_2(\C)$ se relève à $\mathrm{SL}_2(\C)$ puisque ce sont des groupes libres. On pourra ainsi avoir recours à des raisonnements matriciels.

 Dans $\Pd$, le pinceau des droites passant par $A$ est le projeté du pinceau de $\mathbb{P}^1$ de la déformation isomonodromique initiale $(\calR_s)$, en particulier la réduction de la monodromie de $\tilde{\calR}'_{\vert F_b}$ est la même que celle d'un $\calR_s$ (cf section \ref{quotient}).
 Soient $T,U,V,W$ les images respectives de $t,u,v,w$ par un relèvement de $\rho$ à $\mathrm{SL}_2$.

On a vu à la section \ref{monodlin} que, modulo conjugaison globale,  $(V,W,T,U)=\varepsilon\cdot (\phi \cdot(V_0,W_0,T_0,U_0))$ pour un   $\phi \in \mathcal{MCG}(\mathbb{S}_4^2)$ et un changement de signes $\varepsilon \in \{(\varepsilon_i) \in \{\pm 1\}^4 \vert \prod_i \varepsilon_i=1\}$.
De plus, l'action de $\varepsilon \phi$ sur les traces doit être triviale. En particulier, $\phi$ doit induire une permutation de $\{1,2,3,4\}$ de la forme $(23)^k(14)^l$.

Soient $b_{23}$ et $b_{14}$ des éléments de $\mathcal{MCG}(\mathbb{S}_4^2)$ qui induisent respectivement $(23)$ et $(14)$.
Soit $\mathcal{MCG}_{pure}(\mathbb{S}_4^2)$ le noyau de $\mathcal{MCG}(\mathbb{S}_4^2)\rightarrow \textfrak{S}_4$.
L'action de $\mathcal{MCG}(\mathbb{S}_4^2)$  sur les classes de conjugaisons de quadruplets $(V,W,T,U)\in \mathrm{SL}_2$ tels que $TUVW=1$, se décrit par une action polynomiale sur $L=(tr(V),tr(W),tr(T),tr(U),tr(VW),tr(WT),tr(VT))$, cf \cite[p 193]{MR2254812}.
Ainsi, on peut calculer et constater ce qui suit pour le $7$-uplet $L_0$ associé à $(V_0,W_0,T_0,U_0)$.
 $$\mathcal{MCG}_{pure}(\mathbb{S}_4^2)\cdot L_0=b_{14}\mathcal{MCG}_{pure}(\mathbb{S}_4^2)\cdot L_0=\varepsilon^0b_{23}\mathcal{MCG}_{pure}(\mathbb{S}_4^2)\cdot L_0,$$ pour $\varepsilon^0=(+,-,-,+)$; notons que chaque membre compte six éléments, ce qui correspond au degré de notre solution de (PVI).

Ainsi, projectivement et à conjugaison globale près, $$(V,W,T,U) \in \mathcal{MCG}_{pure}(\mathbb{S}_4^2)\cdot (V_0,W_0,T_0,U_0).$$

La monodromie de $\tilde{\calR}'_{\vert E}$ se relève aussi en $\hat{\rho}$ à $\mathrm{SL}_2$. Elle est aisément calculée grâce à la connaissance des exposants : $\theta_{\infty}=0$, $\theta_1=\frac{1}{12}$, $\theta_2=\frac{1}{4}$ correspondant respectivement à $\ell_{\infty}$, $\ell_1$ et $\ell_2$. Soient $A=\hat{\rho}(\alpha)$ et $G=\hat{\rho}(\gamma)$, on a 
$tr(A)=2\cos(\pi\theta_{\infty})$, $tr(G)=2\cos(\pi\theta_1)$ et $tr(AG)=2\cos(\pi\theta_{2})$, ce qui détermine une représentation vers $\mathrm{SL}_2(\C)$, unique  modulo conjugaison et irréductible, d'après \cite{MR1731936}.
Après calcul, elle est donnée par $A= \left[ \begin {array}{cc} 1&0\\\noalign{\medskip}-\sqrt {3}&1
\end {array} \right] $ et \\
$ G= \left[ \begin {array}{cc} \sqrt {2}&\frac{\sqrt {2}}{2}(\sqrt {3}+1)\\\noalign{\medskip}\frac{\sqrt {2}}{2}(5-3\sqrt {3})&\frac{\sqrt {2}}{2}(\sqrt {3}-1)\end {array} \right]$.

En utilisant la deuxième et la quatrième relation de (\ref{relfu}), ainsi que le fait que $T$ soit parabolique, on voit que $T$ et $W$ commutent. Dans l'orbite $\mathcal{MCG}_{pure}(\mathbb{S}_4^2)\cdot (V_0,W_0,T_0,U_0)$, seuls deux quadruplets satisfont cette condition, même projectivement. En utilisant la troisième et la septième relation de (\ref{relfu}), on arrive à voir que l'un d'eux ne peut correspondre à notre représentation  tandis qu'un seul représentant de la classe de conjugaison du second satisfait   (\ref{relfu}), il est décrit ci-dessous. \\
 \begin{minipage}[b]{0.50\linewidth}   
      $U= \left[ \begin{array}{cc} \sqrt {3}-1&1\\\noalign{\medskip}-2+\sqrt{
3}&1\end {array} \right]$; 

   \end{minipage}\hfill
  \hspace{1 cm} \begin{minipage}[b]{0.5\linewidth}    $V=\left[ \begin {array}{cc} -1+2\,\sqrt {3}&1\\\noalign{\medskip}-8+3\,
\sqrt {3}&-\sqrt {3}+1\end {array} \right]$; 
   \end{minipage}
   
     \begin{minipage}[b]{1\linewidth}   
    \hspace{3.3 cm} $W= \left[ \begin {array}{cc} 1&0\\\noalign{\medskip}-\sqrt {3}-1&1
\end {array} \right]$.
\end{minipage}

Le groupe $<A,G,U,V,W>$ est un sous-groupe de $\mathrm{PSL}_2(\mathbb{Z}[\xi])$ où $\xi=2 \cos(\frac{\pi}{12})=\frac{\sqrt{2}}{2}(1+\sqrt{3})$. Notons que $\mathbb{Z}[\xi]$ est l'anneau des entiers de $\mathbb{Q}[\xi]=\mathbb{Q}[\sqrt{2},\sqrt{3}]$.
On remarque aussi que la représentation $\rho$ ne se relève pas à $\mathrm{SL}_2$.

Exactement la même démarche permet de déterminer la monodromie de la structure transverse du feuilletage dual $\G$: vue l'involution, les images de $v$, $w$, $t$, $u$ sont encore donnés par un élément de l'orbite 
$\mathcal{MCG}_{pure}(\mathbb{S}_4^2)\cdot(V_0,W_0,T_0,U_0)$ et une seule trace de monodromie locale change : $\theta_1=\frac{7}{12}$. La représentation obtenue est l'image de la précédente par $\sqrt{3} \mapsto -\sqrt{3}$.

Un examen plus approfondi donne le résultat suivant.
\begin{lemm}\label{monodromie}
\begin{enumerate}
\item
Le sous-groupe $\Gamma$ de $\mathrm{PSL}_2(\C)$  engendré par $U$, $V$, $W$, $A$ et $G$ est une extension de degré $2$ de $\Gamma_{\sqrt{3}}=\mathrm{PSL}_2(\mathbb{Z}[\sqrt{3}])$.

\item Ce groupe coïncide avec le groupe modulaire étendu $\mathrm{PGL}^+(\mathbb{Z}[\sqrt{3}])$ et, par le plongement habituel, c'est un sous groupe discret maximal de $\mathrm{PGL}^+(\R)^2$.
\end{enumerate}
\begin{proof}
\begin{enumerate}
\item Le groupe $\mathrm{PSL}_2(\mathbb{Z}[\sqrt{3}])$ est engendré par ses deux sous-groupes de matrices triangulaires : $T_1:=\left \{ \left [ \begin{array}{cc} 1 & 0\\ u & 1 \end{array}   \right ], u \in \mathbb{Z}[ \sqrt{3}] \right \} $ et
$ T_2:=\left \{ \left [ \begin{array}{cc} 1 & u\\ 0 & 1 \end{array}   \right ], u \in \mathbb{Z}[ \sqrt{3}] \right \} $.
L'isomorphisme  $ \left [ \begin{array}{cc} 1 & 0\\ u & 1 \end{array}   \right ] \mapsto u$ permet d'identifier  $T_1$ à $\mathbb{Z}[\sqrt{3}]$, on voit ainsi facilement que $T_1$ est engendré par les matrices $A$ et $W$. De même, les matrices $H=G^2AG^{-2}=\left [ \begin{array}{cc} 1 & 12+ 7\sqrt{3}\\ 0 & 1 \end{array}   \right ]$ et $F=G^2WG^{-2}=\left [ \begin{array}{cc} 1 & 19+ 11\sqrt{3}\\ 0 & 1 \end{array}   \right ]$ engendrent $T_2$, puisque  $12+ 7\sqrt{3}$ et $19+ 11\sqrt{3}$ sont des générateurs du $\Z$-module $\mathbb{Z}[\sqrt{3}]$ en vertu des identités suivantes :

$1=-11(12+7\sqrt{3})+7(19+ 11\sqrt{3}) ; \sqrt{3}=19(12+7\sqrt{3})-12(19+ 11\sqrt{3}).$
Ainsi, le groupe $<A,G,U,V,W>$ contient $\mathrm{PSL}_2(\mathbb{Z}[\sqrt{3}])$.

 D'autre part, $G$ est le seul générateur non contenu dans $\mathrm{PSL}_2(\mathbb{Z}[\sqrt{3}])$ tandis que $G^2$  l'est, ce qui montre que $<A,G,U,V,W>/\mathrm{PSL}_2(\mathbb{Z}[\sqrt{3}])$ se réduit à deux éléments.

\item Dans un but de concision, on se réfère librement à \cite[I.4]{MR930101}. Dans $\mathrm{PGL}_2(\R)$, la matrice $G$  donne le même élément que la matrice
 $\xi \cdot G$ qui est donnée ci-dessous.

$$ \left[ \begin {array}{cc} \sqrt {3}+1&2+\sqrt {3}
\\ \noalign{\medskip}\sqrt {3}-2&1\end {array} \right] $$
Cette dernière a pour déterminant $\xi^2\in \mathbb{Z}[\sqrt{3}]$ qui est totalement positif, donc $\Gamma$ est un sous groupe de $\hat{\Gamma}_{\sqrt{3}}=\mathrm{PGL}^+(\mathbb{Z}[\sqrt{3}])$.
De plus, $\Gamma_{\sqrt{3}}$ est d'indice $2$ dans $\hat{\Gamma}_{\sqrt{3}}$, \cite[fin de la p. 11]{MR930101}; comme on a les inclusions

$\Gamma_{\sqrt{3}}<\Gamma<\hat{\Gamma}_{\sqrt{3}}$, on en déduit $\Gamma=\hat{\Gamma}_{\sqrt{3}}$.
D'après \cite{MR930101}, il existe, parmi les sous groupes discrets de $\mathrm{PGL}^+(\R)^2$ contenant $\hat{\Gamma}_{\sqrt{3}}$, un élément maximal $\Gamma_{HM}$ : l'extension de Hurwitz-Maass de $\Gamma_{\sqrt{3}}$. L'indice de $\Gamma_{\sqrt{3}}$ dans $\Gamma_{HM}$ est donné par $2^{t-1}$ où $t$ est le nombre de facteurs premiers distincts du discriminant $D=12$ de $\Q(\sqrt{3})$, cf \cite[p $13$]{MR930101}.
Ainsi $2^{t-1}=2$ et $\Gamma_{HM}$ coïncide avec $\hat{\Gamma}_{\sqrt{3}}$.

\end{enumerate}
\end{proof}
\end{lemm}

\subsection{La surface modulaire associée à $\Z[\sqrt{3}]$}
On déduit de notre travail suffisamment d'informations sur la surface $Z_{\sqrt{3}}$.
\begin{theo}\label{racine}
Un modèle birationnel de la surface modulaire bifeuilletée $(Z_{\sqrt{3}},\F_{\sqrt{3}},\G_{\sqrt{3}})$ est $(\Pd,\F_{\omega_1},\G_{\tau_1})$
où \\ \footnotesize
$\omega_1=6\, \left( 3\,{v}^{2}+1 \right) v \left( {v}^{2}+9\,u{v}^{2}+3\,u
 \right) du \\+ \left(  \left( 9\,u-5 \right)  \left( 9
\,u-2 \right)  \left( 9\,u-1 \right) {v}^{4}+9\,u \left( 5+54\,{u}^{2}
-30\,u \right) {v}^{2}+9\,{u}^{2} \left( 9\,u-2 \right)  \right) dv.$

$\tau_1=6\, \left( 3\,{v}^{2}+1 \right) v \left( -8\,{v}^{2}-3+36\,u{v}^{2}+12
\,u \right) du \\+ \left(  \left( 9\,u-5 \right) 
 \left( 9\,u+1 \right)  \left( 9\,u-1 \right) {v}^{4}+ \left( 3+486\,{
u}^{3}-432\,{u}^{2}+45\,u \right) {v}^{2}+9\,u \left( 9\,u-2 \right) 
 \left( u-1 \right)  \right) dv.$
\normalsize

De plus,  $\sigma_1 : (u,v)\mapsto  \left( {\frac {3\,{v}^{2} \left( 36\,{v}^{2}+13 \right) u-{v}^{2} \left( 20\,{v}^{2}+9 \right) }{9\, \left( 12\,{v}^{2}-1 \right) \left( 3\,{v}^{2}+1 \right)u  -3\,{v}^{2} \left( 36\,{v}^{2}+13 \right) }},v \right)$ est une involution birationnelle de $\Pd$ qui échange $\F_{\omega_{1}}$ et $\G_{\tau_{1}}$.
\begin{proof}
D'après le calcul de la monodromie des structures transverses de $\F$ et $\G$ et le corollaire \ref{coro1}, on voit que l'on obtient la surface $(Y_{\sqrt{3}},\F_{\sqrt{3}},\G_{\sqrt{3}})$ comme revêtement double de  $(Y_{\Gamma},\F_{\Gamma},\G_{\Gamma})$. 
Soit $\pi : Y_{\sqrt{3}}\rightarrow Y_{\Gamma}$ le revêtement en question.

Soit $X$ l'éclatement de $\Pd$ en $A$ et $E$ le diviseur exceptionnel, comme dans la section \ref{monod}. 
La fin de la désingularisation de $\F$ donne un morphisme $\phi : Z_{\Gamma} \rightarrow X$ et on note $ \psi : Y_{\Gamma} \tilde{\dasharrow} Z_{\Gamma}$ la désingularisation de $Y_{\Gamma}$.
Soient $P$ le lieu invariant du feuilletage induit par $\F$ sur $X$, $X^1=X\setminus (P\cup E)$, $Y_{\sqrt{3}}^1=Y_{\sqrt{3}} \setminus (\pi \circ \psi \circ \phi)^*(P\cup E)$.

La composée $\pi \circ \psi \circ \phi$ induit un revêtement étale $\pi^1 : Y_{\sqrt{3}}^1 \rightarrow X^1$ qui est déterminé par sa représentation de monodromie : $$ \pi_1(X^1)\stackrel{\varepsilon} {\rightarrow}\Gamma / \mathrm{PSL}_2(\Z[\sqrt{3}])=\Z/2\Z.$$

 D'après le lemme \ref{monodromie}, on voit que $ p : \Pu \times \Pu \dasharrow X,(u,v)\mapsto (u,\frac{-v^2}{3v^2+1})$ induit en restriction à $X^1$ un revêtement étale qui a exactement $\varepsilon$ pour monodromie.
Ainsi on a un biholomorphisme $T$ entre $Y_{\sqrt{3}}^1$ et $\Pu \times \Pu\setminus p^*(P\cup E)$, tel que $p \circ T=\pi^1$. Comme $p^{-1}\circ \pi^1$ est une correspondance algébrique, on en déduit que $T$ se prolonge en une transformation birationnelle, de sorte que le diagramme suivant commute.

$$\xymatrix{Y_{\sqrt{3}} \ar@{-->}[rrd]^{\pi^1} \ar[d]_\pi \ar@{-->}[rr]^T& & \Pu \times \Pu    \ar@{-->}[d]^p  \\
Y_{\Gamma} \ar@{-->}[r]_\psi &Z_{\Gamma}  \ar[r]_\phi & X }$$
Ainsi, en pratique, le revêtement double qu'on utilise est $p$, ce qui permet de déduire le résultat.   
\end{proof}
\end{theo}
Notons que $Z_{\sqrt{3}}$ est un "revêtement" intermédiaire entre le $\Pu \times \Pu$ de la déformation isomonodromique initiale $\calR$ et $Z_{\Gamma}$ puisque
$v(s)={\frac {-2\sqrt{-1}s}{{s}^{2}+3}}$ satisfait \scalebox{0.87}{$\frac{-v(s)^2}{3v(s)^2+1}={\frac {4{s}^{2}}{ \left( {s}^{2}-3 \right) ^{2}}}=r(s)$.}

\section{Conclusion}
Dans une perspective plus générale, on peut se demander quelles solutions (algébriques) de l'équation de Painlevé VI donnent des exemples intéressants de feuilletages. 
L'intérêt qu'on porte à un feuilletage de Riccati $\calR$ sur $\Pu \times X \rightarrow X$ peut provenir de la richesse de son groupe de monodromie (disons sa Zariski densité) et du fait qu'il ne soit pas birationnellement équivalent à un tiré en arrière d'un feuilletage de Riccati au dessus d'une courbe (notamment en vue du résultat susmentionné de Corlette-Simpson). 

Signalons que notre exemple montre que l'action du groupe d'Okamoto perturbe ces deux propriétés. En effet, la monodromie du feuilletage de Riccati $\mathcal{H}$ associé à la solution tétrahédrale n°$6$ de Boalch est un groupe fini et, pour cette raison, $\mathcal{H}$ est le tiré en arrière d'un feuilletage de Riccati au dessus d'une courbe.

Dans le cas où $\calR$ n'est pas obtenu par tiré en arrière, Corlette et Simpson montrent que la monodromie de $\calR$ est définie sur un anneau d'entiers et l'estimation du degré de son corps de fraction permet d'estimer la dimension d'un "polydisk Shimura D-M stack" par lequel la représentation factorise.
 Dans un prochain travail, on donnera une étude systématique de ces propriétés pour les solutions algébriques de la liste \cite{LiTy} et leurs transformations d'Okamoto. 

\bibliographystyle{smfalpha}
\bibliography{mabiblio}

\end{document}